\documentclass[english]{article}

\usepackage{algorithm2e}
\usepackage{amsthm}
\usepackage{amstext}
\usepackage{amssymb}
\usepackage{amsmath}
\usepackage{array}

\newcommand{\Text}[1]{\text{\textnormal{#1}}}
\newcommand{\ld}{\Text{ld}}
\newcommand{\Tr}{\Text{Tr}}

\title{Computing the log-determinant of symmetric, diagonally dominant matrices
in near-linear time}

\author{Timothy Hunter\thanks{Department of Electrical Engineering and Computer Sciences, University of California, Berkeley.} \and 
Ahmed {El Alaoui}$^{*}$ \and
Alexandre M. Bayen$^{*}$}

\newtheorem{theorem}{Theorem}
\newtheorem{definition}{Definition}
\newtheorem{lemma}{Lemma}
\newtheorem{proposition}{Proposition}
\newtheorem{corollary}{Corollary}

\begin{document}
\maketitle

\begin{abstract}
We present new algorithms for computing the log-determinant of symmetric,
diagonally dominant matrices. Existing algorithms run with cubic complexity
with respect to the size of the matrix in the worst case. Our algorithm 
computes an approximation of the log-determinant in time near-linear with 
respect to the number of non-zero entries and with high probability. This 
algorithm builds upon the utra-sparsifiers introduced by Spielman and Teng for
Laplacian matrices and ultimately uses their refined versions introduced by Koutis, Miller and Peng in the context of solving linear systems. We also present simpler algorithms that compute upper and 
lower bounds and that may be of more immediate practical interest.
\end{abstract}

\footnotetext{Under submission to the SIAM Journal on Computing}

\section{Introduction}

We consider the problem of computing the determinant of symmetric,
diagonally dominant (SDD) matrices, i.e. real symmetric matrices $A$
for which: 
\[
A_{ii}\geq\sum_{j\neq i}\left|A_{ij}\right|
\]
The set of all such matrices of size $n\times n$ is denoted $SDD_{n}$,
and the set of all symmetric real matrices is called $\mathcal{S}_{n}$.
Call $m$ the number of non-zero entries in $A$. We are interested
in computing the determinant of sparse matrices, i.e. matrices for
which $m\ll n^{2}$.

The best exact algorithm known for computing the determinant of general
matrices, the Cholesky factorization, runs in a cubic complexity $\mathcal{O}\left(n^{3}\right)$.
Computing the factorization can be sped up for a few specific patterns
such as trees, but no algorithm has been shown to work in a generic
way for $SDD_{n}$, let alone general symmetric matrices. We present
an algorithm that returns an approximation of the logarithm of the
determinant in time quasi-linear with the number of non-zero entries
of $A$. More specifically, we show that our algorithm, \texttt{UltraLogDet},
computes an $\epsilon$-approximation of the logarithm of the determinant
with high probability and in expected time%
\footnote{We use the notation $\tilde{\mathcal{O}}$ to hide a factor at most
$\left(\log\log n\right)^{8}$%
}: 
\[
\tilde{\mathcal{O}}\left(m\epsilon^{-2}\log^{3}n\log^{2}\left(\frac{n\kappa_{A}}{\epsilon}\right)\right)
\]

where $\kappa_{A}$ is the condition number of $A$. This algorithm
builds upon the work of Spielman and Teng on \emph{ultra-sparsifiers}
\cite{Spielman2009a}, and it critically exploits the recent improvements
from Koutis, Miller and Peng \cite{Koutis2010}. This is to our knowledge
the first algorithm that presents a nearly linear complexity which
depends neither on the condition number of $A$ (except through a
log-term) nor on a specific pattern for the non-zero coefficients
of $A$.

The high sophistication of the algorithm transpires through the large
exponent of $\log\log n$. However, our algorithm will directly benefit
from any improvement on ultra-sparsifiers. Given the considerable
practical importance of such preconditioners, we expect some fast
improvements in this area. Also, the bulk of the work is performed
in a Monte Carlo procedure that is straightforward to parallelize.
Furthermore, we also present simpler, non-optimal algorithms that
compute upper and lower bounds of the logarithm of the determinant,
and that may be of more immediate practical interest.

\subsection{Background}

There are two approaches in numerical linear algebra to approximately
compute a determinant (or the log of the determinant): by performing
a (partial) Cholesky factorization of $A$, or by considering the
trace of some power series.

As mentioned above, the Cholesky factorization performs a decomposition
of the form: $A=PLDL^{T}P^{T}$ with $P$ a permutation matrix, $L$
a low-triangular matrix with $1$ on the diagonal and $D$ a diagonal
matrix of non-negative coefficients. Then the log-determinant of $A$
is simply%
\footnote{We will use the $\left|\cdot\right|$ operator to denote the determinant,
it will be clear from the context that it is different from the absolute
value.%
}: 
\[
\log\left|A\right|=\sum_{i}\log D_{ii}
\]
The complexity of dense Cholesky factorization for dense matrices
is $\mathcal{O}\left(n^{3}\right)$. Unfortunately, Cholesky factorization
usually does not gain much from the knowledge of the sparsity pattern
due to the \emph{fill-in problem} (see \cite{meurant1999computer},
section 3.2). There is one case, though, for which Cholesky factorization
is efficient: if the sparsity pattern of $A$ is a tree, then performing
Cholesky factorization takes $\mathcal{O}\left(n\right)$ time, and
the matrix $L$ is a banded matrix \cite{liu1990eliminationtrees}.
If the sparsity pattern of $A$ is not a tree, however, this advantageous
decomposition does not hold anymore.

When the matrix $A$ is close to the identity, more precisely when
the spectral radius of $M=A-I$ is less than~$1$, one can use the
remarkable Martin expansion of the log-determinant \cite{martin1992approximations}:
\begin{equation}
\log\left|A\right|=\text{Tr}\left(\log A\right)\label{eq:martin-expansion}
\end{equation}
where $\log A$ is the matrix logarithm defined by the series expansion:
\begin{equation}
\log A=\sum_{i=0}^{\infty}\frac{\left(-1\right)^{i}}{i+1}M^{i}\label{eq:matrix-log}
\end{equation}
The determinant can then be computed by a sum of traces of the power
of $M$, and the rate of convergence of this series is driven by the
spectral radius $M$. This line of reasoning has led researchers to
look for decompositions of $A$ of the form $A=U+V$ with the determinant
of $U$ being easier to compute and $U^{-1}V+I$ having a small spectral
radius. Then $\log\left|A\right|=\log\left|U\right|+\log\left|U^{-1}V+I\right|$.
The most common decomposition $U,V$ is in terms of block diagonal
and off-diagonal terms, which can then use Hadamard inequalities on
the determinant to bound the error \cite{Ipsen2006}. Diagonal blocks
also have the advantage of having determinants easy to compute. However,
this approach requires some strong assumptions on the condition number
of $A$, which may not hold in practice.

The trace approach is driven by \emph{spectral properties }(the condition
number) while the Cholesky approach is driven by \emph{graphical }properties\emph{
}(the non-zero pattern)\emph{. }We\emph{ }propose to combine these
two approaches by decomposing the problem with one component that
is close to a tree (and is more amenable to Cholesky methods), and
one component that has a bounded condition number. Our solution is
to use a \emph{spectral sparsifier} introduced by Spielman in \cite{Spielman2008}.

\subsection{Applications}

The problem of estimating determinants has important applications
in spatial data analysis, statistical physics and statistics. In spatial
statistics, it is often convenient to interpolate measurements in
a 2-, 3- or 4-dimensional volume using a sparse Gaussian process,
a technique known in the geospatial community as \emph{kriging }\cite{zhang2010kriging,li2005analysis}\emph{.
}Computing the optimal parameters of this Gaussian process involves
repeated evaluations of the partition function, which is a log-determinant.
In this context, a diagonally dominant matrix for the Gram matrix
of the process corresponds to distant interactions between points
of measure (which is verified in some contexts, see \cite{KelleyPace1997291}).
Determinants also play a crucial role in quantum physics and in theoretical
physics. The wave function of a system of multiple fermion particles
is an antisymmetric function which can be described as a determinant
(Slatter determinant, \cite{atkins2011molecular,lowdin1955quantum}).
In the theory of quantum chromodynamics (QCD), the interaction between
particles can be discretized on a lattice, and the energy level of
particles is the determinant of some functional operators over this
lattice \cite{duncan1998efficient}. It is itself a very complex problem
because of the size of the matrices involved for any non-trivial problem,
for which the number of variables is typically in the millions \cite{bernardson1994monte}.
In this setting, the restriction to diagonally dominant matrices can
be interpreted as an interaction between relatively massive particles
\cite{deForcrand1989516}, or as a bound on the propagation of interactions
between sites in the lattice \cite{bernardson1994monte}.

For these reasons, computing estimates of the log-determinant has
been an active problem in physics and statistics. In particular, the
Martin expansion presented in Equation \eqref{eq:martin-expansion}
is extensively used in quantum physics \cite{Ipsen2006}, and it can
be combined with sampling method to estimate the trace of a matrix
series (\cite{Zhang2008},\cite{McCourt2008},\cite{Zhang2007}).
Another different line of research has worked on bounds on the values
of the determinant itself. This is deeply connected to simplifying
statistical models using variational methods. Such a relaxation using
a message-passing technique is presented in \cite{Wainwright2006}.
Our method is close in spirit to Reuksen's work \cite{Reusken2002}
by the use of a preconditioner. However, Reuksen considers preconditioners
based on a clever approximation of the Cholesky decomposition, and
its interaction with the eigenvalues of the complete matrix is not
well understood. Using simpler methods based on sampling, we are able
to carefully control the spectrum of the remainder, which in turn
leads to strong convergence guarantees.

\subsection{A note on scaling}

Unlike other common characteristics of linear operators, the determinant
and the log-determinant are very sensitive to dimensionality. We will
follow the approach of Reuksen \cite{Reusken2002} and consider the
\emph{regularized log-determinant} $f\left(A\right)=n^{-1}\log\left|A\right|$
instead of the log-determinant. The regularized determinant has appealing
properties with respect to dimensionality. In particular, its sensitivity
to perturbations does not increase with the dimensionality, but only
depends on spectral properties of the operator $A$. For example,
calling $\lambda_{\min}$ and $\lambda_{\max}$ the minimum and maximum
eigenvalues of $A$, respectively: 
\[
\log\lambda_{\min}\leq f\left(A\right)\leq\log\lambda_{\max}
\]
\[
\left|f\left(A+\epsilon I\right)-f\left(A\right)\right|\leq\epsilon\left\Vert A^{-1}\right\Vert _{2}+\mathcal{O}\left(\epsilon^{2}\right)
\]
The last inequality in particular shows that any perturbation to $\log\left|A\right|$
will be in the order $\mathcal{O}\left(n\right)$, and so that all
the interesting log-determinants in practice will be dominated by
some $\mathcal{O}\left(n\right)$.

\subsection{Main results}

We first present some general results about the preconditioning of
determinants. Consider $A\in SDD_{n}$ invertible, and some other
matrix $B\in SDD_{n}$ that is close to $A$ in the spectral sense.
All the results of this article stem from observing that: 
\begin{eqnarray*}
\log\left|A\right| & = & \log\left|B\right|+\log\left|B^{-1}A\right|\\
 &  & \log\left|B\right|+\text{Tr}\left(\log\left(B^{-1}A\right)\right)
\end{eqnarray*}
The first section is concerned with estimating the remainder term
$\text{Tr}\left(\log\left(B^{-1}A\right)\right)$ using the Martin
expansion. The exact inverse $B^{-1}$ is usually not available, but
we are given instead a linear operator $C$ that is an $\epsilon-$approximation
of $B^{-1}$, for example using a conjugate gradient method. We show
in Section \ref{sec:Preconditioned-log-determinants} that if the
precision of this approximation is high enough, we can estimate the
remainder with high probability and with a reasonable number of calls
to the operator $C$ (this sentence will be made precise in the rather
technical Theorem \ref{thm:preconditioning-approx}). Using this general
framework, the subsequent Section \ref{sec:Making-the-problem} shows
that spectral sparsifiers make excellent preconditioners that are
close enough to $A$ and so that computing the Martin expansion is
not too expansive. In particular, we build upon the recursive structure
of Spielman-Teng ultra-sparsifiers to obtain our main result: 
\begin{theorem}
\label{thm:ultra_main}On input $A\in SDD_{n}$ with $m$ non-zeros,
$\eta>0$, the algorithm \texttt{UltraLogDet} returns a scalar $z$
so that: 
\[
\mathbb{P}\left[\left|z-n^{-1}\log\left|A\right|\right|>\epsilon\right]\leq\eta
\]
and this algorithm completes in expected time $\tilde{O}\left(m\epsilon^{-2}\log^{3}n\log^{2}\left(\frac{\kappa_{A}}{\epsilon}\right)\log\left(\eta^{-1}\right)\right)$.
Moreover, if $\epsilon>\Omega(n^{-1})$, then the running time improves
by a factor $\epsilon$. 
\end{theorem}

The rest of the article is structured as follows. In the next section,
we present some results about estimating the log-determinant from
a truncated expansion. These results will justify the use of \emph{preconditioners
}to compute the determinant of a matrix. The techniques developed
by Spielman et al. work on the Laplacians of weighted graphs. Section
3 introduces some new concepts to expand the notion of determinants
to Laplacian matrices, and presents a few straightforward results
in the relations between graph Laplacians and SDD matrices. Section
3.2 will use these new concepts to introduce a first family of preconditioners
based on low-stretch spanning trees. Finally, Section 3.3 contains
the proof of our main result, an algorithm to compute determinants
in near-linear time.

\section{Preconditioned log-determinants\label{sec:Preconditioned-log-determinants}}

We begin by a close inspection of a simple sampling algorithm to compute
log-determinants, presented first in \cite{Barry1999}. We will first
present some error bounds on this algorithm that expand on bounds
previously presented in \cite{Bai1996} and \cite{Barry1999}. This
section considers general symmetric matrices and does not make assumptions
about diagonal dominance.

Consider a real symmetric matrix $S\in\mathcal{S}_{n}^{+}$ such that
its spectral radius is less than $1$: $0\preceq S\preceq\left(1-\delta\right)I$
for some $\delta\in\left(0,1\right)$. Our goal is to compute $\log\left|I-S\right|$
up to precision $\epsilon$ and with high probability. From the Martin
expansion: 
\begin{equation}
\log\left|I-S\right|=-\Tr\left(\sum_{k=1}^{\infty}\frac{1}{k}S^{k}\right)\label{eq:martin}
\end{equation}

This series of traces can be estimated by Monte Carlo sampling, up
to precision $\epsilon$ with high probability, by truncating the
series and by replacing the exact trace evaluation by $x^{T}S^{k}x$
for some suitably chosen random variables $x$. In order to bound
the errors, we will bound the large deviation errors using the following
Bernstein inequality: 
\begin{lemma}
{[}Bernstein's inequality{]} \label{lem:bernstein} Let $X_{1}\cdots X_{n}$
be independent random variables with $\mathbb{E}\left[X_{i}\right]=0$,
$\left|X_{i}\right|<c$ almost surely. Call $\sigma^{2}=\frac{1}{n}\sum_{i}\text{\textnormal{Var}}\left(X_{i}\right)$,
then for all $\epsilon>0$: 
\[
\mathbb{P}\left[\frac{1}{n}\Big|\sum_{i}X_{i}\Big|\geq\epsilon\right]\leq2\exp\left(-\frac{n\epsilon^{2}}{2\sigma^{2}+2c\epsilon/3}\right)
\]

\end{lemma}
We can adapt some results from \cite{Barry1999} to prove this bound
on the deviation from the trace. 
\begin{lemma}
\label{lem:bernstein-trace}Consider $H\in\mathcal{S}_{n}$ with the
assumption $\lambda_{\text{\textnormal{min}}}I_{n}\preceq H\preceq\lambda_{\text{\textnormal{max}}}I$.
Consider $p$ vectors sampled from the standard Normal distribution:
$\mathbf{u}_{i}\sim\mathcal{N}\left(\mathbf{0},I_{n}\right)$ for
$i=1\cdots p$. Then for all $\epsilon>0$: 
\[
\mathbb{P}\left[\left|\frac{1}{p}\sum_{i=1}^{p}\frac{\mathbf{u}_{i}^{T}H\mathbf{u}_{i}}{\mathbf{u}_{i}^{T}\mathbf{u}_{i}}-\frac{1}{n}\Tr\left(H\right)\right|\geq\epsilon\right] \leq 2\exp\left(-\frac{p\epsilon^{2}}{4\frac{\left(\lambda_{\max}-\lambda_{\min}\right)^{2}}{n}+2\frac{\left(\lambda_{\max}-\lambda_{\min}\right)\epsilon}{3}}\right)
\]
\end{lemma}
\begin{proof}
The distribution of $\mathbf{u}_{i}$ is invariant through a rotation,
so we can consider $H$ diagonal. We assume without loss of generality
that $H=\text{diag}\left(\lambda_{1},\cdots,\lambda_{n}\right)$.
Again without loss of generality, we assume that $\lambda'_{\max}=\lambda_{\max}-\lambda_{\min}$
and $\lambda'_{\min}=0$ (by considering $H'=H-\lambda_{\min}I$).
Call $V_{i}=\frac{\mathbf{u}_{i}^{T}H\mathbf{u}_{i}}{\mathbf{u}_{i}^{T}\mathbf{u}_{i}}-n^{-1}\text{Tr}\left(H\right)$.
Using results from \cite{Barry1999}, we have: $\left|V_{i}\right|\leq\lambda_{\max}-\lambda_{\min}$,
$\mathbb{E}\left[V_{i}\right]=0$ and 
\[
\text{Var}\left(V_{i}\right)=\frac{2}{n(n+2)}\sum_{i=1}^{n}\left(\lambda_{i}-n^{-1}\text{Tr}\left(H\right)\right)^{2}
\]

Each of the variables $V_{i}$ is independent, 
so invoking Lemma \ref{lem:bernstein} gives: 
\[
\mathbb{P}\left[\frac{1}{p}\left|\sum_{i=1}^{p}V_{i}\right|\geq\epsilon\right]\leq2\exp\left(-\frac{p\epsilon^{2}}{2\sigma^{2}+2\left(\lambda_{\max}-\lambda_{\min}\right)\epsilon/3}\right)
\]
with 
\begin{align*}
\sigma^{2} & =\frac{2}{n(n+2)}\sum_{i=1}^{n}\left(\lambda_{i}-n^{-1}\text{Tr}\left(H\right)\right)^{2}\\
 & \leq\frac{2}{n^{2}}\sum_{i=1}^{n}\left(\lambda_{\max}-\lambda_{\min}\right)^{2}=\frac{2}{n}\left(\lambda_{\max}-\lambda_{\min}\right)^{2}
\end{align*}

\end{proof}
The previous lemma shows that if the eigenspectrum of a matrix is
bounded, we can obtain a Bernstein bound on the error incurred by
sampling the trace. Furthermore, the convergence of the series \eqref{eq:martin}
is also determined by the extremal eigenvalues of $S$. If we truncate
the series (\ref{eq:martin}), we can bound the truncation error using
the extremal eigenvalues. We formalize this intuition in the following
theorem, which is adapted from the main theorem in \cite{Barry1999}.
While that main theorem in \cite{Barry1999} only considered a confidence
interval based on the covariance properties of Gaussian distribution,
we generalize this result to a more general Bernstein bound. 
\begin{theorem}
\label{thm:det-sampling-theorem}Consider $S\in\mathcal{S}_{n}^{+}$
with $0\preceq S\preceq\left(1-\delta\right)I$ for some $\delta\in\left(0,1\right)$.
Call $y=n^{-1}\log\left|I-S\right|$ the quantity to estimate, and
consider $\mathbf{u}_{i}\sim\mathcal{N}\left(\mathbf{0},I_{n}\right)$
for $i=1\cdots p$ all independent. Call $\hat{y}_{p,l}$ an estimator
of the truncated series of $l$ elements computed by sampling the
trace using $p$ samples: 
\[
\hat{y}_{p,l}=-\frac{1}{p}\sum_{j=1}^{p}\sum_{k=1}^{l}\frac{1}{k}\frac{\mathbf{u}_{j}^{T}S^{k}\mathbf{u}_{j}}{\mathbf{u}_{j}^{T}\mathbf{u}_{j}}
\]
Given $\epsilon>0$ and $\eta\in\left(0,1\right)$, the $\hat{y}_{p,l}$
approximates $y$ up to precision $\epsilon$ with probability at
least $1-\eta$ by choosing $p\geq16\left(\frac{1}{\epsilon}+\frac{1}{n\epsilon^{2}}\right)\log\left(2/\eta\right)\log^{2}\left(\delta^{-1}\right)$
and $l\geq2\delta^{-1}\log\left(\frac{n}{\delta\epsilon}\right)$:
\[
\mathbb{P}\left[\left|y-\hat{y}_{p,l}\right|\geq\epsilon\right]\leq\eta
\]

\end{theorem}
The proof of this result is detailed in Appendix A.

From this theorem we derive two results that justify the notion of
preconditioners for determinants: one for exact preconditioners and
one for approximate preconditioners. The corresponding algorithm,
which we call \texttt{Precond\-itioned\-LogDetMonteCarlo}, is presented
in Algorithm \ref{alg:SampleLogDet}. 
\begin{corollary}
\label{cor:preconditioning}Let $A\in\mathcal{S}_{n}^{+}$ and $B\in\mathcal{S}_{n}^{+}$
be positive definite matrices so that $B$ is a $\kappa-$approximation
of $A$: 
\begin{equation}
A\preceq B\preceq\kappa A\label{eq:A-B-bounds}
\end{equation}
Given $\epsilon>0$ and $\eta\in\left(0,1\right)$, the algorithm
\texttt{PreconditionedLogDetMonteCarlo} computes $\frac{1}{n}\log\left|B^{-1}A\right|$
up to precision $\epsilon$ with probability greater than $1-\eta$,
by performing $16\kappa\left(\frac{1}{\epsilon}+\frac{1}{n\epsilon^{2}}\right)\log\left(\frac{2\kappa}{\epsilon}\right)\log\left(2/\eta\right)\log^{2}\left(\kappa\right)$
vector inversions from $B$ and vector multiplies from $A$. 
\end{corollary}

The proof of this corollary is presented in Appendix A. Usually, computing
the exact inverse by an SDD matrix is too expensive. We can instead
extend the previous result to consider a black box procedure that
approximately computes $B^{-1}x$. If the error introduced by the
approximate inversion is small enough, the result from the previous
corollary still holds. This is what the following theorem establishes: 
\begin{theorem}
\label{thm:preconditioning-approx}Consider $A,B\in\mathcal{S}_{n}^{+}$
positive definite with $B$ a $\kappa-$approximation of $A$ with
$\kappa\geq2$. Furthermore, assume there exists a linear operator
$C$ so that for all $y\in\mathbb{R}^{n}$, $C$ returns a $\nu-$approximation
of $B^{-1}y$: 
\[
\left\Vert C\left(y\right)-B^{-1}y\right\Vert _{B}\leq\nu\left\Vert B^{-1}y\right\Vert _{B}
\]
Given $\eta\in\left(0,1\right)$ and $\epsilon>0$, if $\nu\leq\min\left(\frac{\epsilon}{8\kappa^{3}\kappa\left(B\right)},\frac{1}{2\kappa}\right)$,
then the algorithm \\
 \texttt{PreconditionedLogDetMonteCarlo} returns a scalar $z$ so
that: 
\[
\mathbb{P}\left[\left|z-n^{-1}\log\left|B^{-1}A\right|\right|\geq\epsilon\right]\leq\eta
\]
by performing $64\kappa\left(\frac{1}{\epsilon}+\frac{1}{n\epsilon^{2}}\right)\log\left(\frac{2\kappa}{\epsilon}\right)\log\left(2/\eta\right)\log^{2}\left(\kappa\right)$
vector calls to the operator $C$ and vector multiplies from $A$. 
\end{theorem}

The proof of this result is detailed in Appendix A. While the overall
bound looks the same, the constant (taken away by the $\mathcal{O}\left(\cdot\right)$
notation) is four times as large as in Corollary \ref{cor:preconditioning}.

This last theorem shows that we can compute a good approximation of
the log-determinant if the preconditioner $B$: (a) is close to $A$
in the spectral sense, and (b) can be approximately inverted and the
error introduced by the approximate inversion can be controlled. This
happens to be the case for symmetric, diagonally dominant matrices.

\begin{algorithm}
Algorithm \textbf{PreconditionedLogDetMonteCarlo}($B$,$A$,$\eta$,$p$,$l$):

$y\leftarrow0$

for $j$ from $1$ to $p$:

~~Sample $\mathbf{u}\sim\mathcal{N}\left(\mathbf{0},I\right)$

~~$\mathbf{v}\leftarrow\mathbf{u}/\left\Vert \mathbf{u}\right\Vert $

~~$z\leftarrow0$

~~for $k$ from $1$ to $l$:

~~~~$\mathbf{v}\leftarrow B^{-1}A\mathbf{v}$~up to precision
$\eta$

~~~~$z\leftarrow z+k^{-1}\mathbf{v}^{T}\mathbf{u}$

~~$y\leftarrow y+p^{-1}z$

Return $y$

\caption{PreconditionedLogDetMonteCarlo\label{alg:SampleLogDet}}
\end{algorithm}

\section{Ultra-sparsifiers as determinant preconditioners}

\subsection{Reduction on a Laplacian}

\label{sec:Making-the-problem}

From now on, we consider the computation of $\log A$, where $A\in SDD_{n}$.
The techniques we will develop work on Laplacian matrices instead
of SDD matrices. An SDD matrix is positive semi-definite while a Laplacian
matrix is always singular, since its nullspace is spanned by $\mathbf{1}$.
We generalize the definition of the determinant to handle this technicality. 
\begin{definition}
\emph{Pseudo-log-determinant (PLD):} Let $A\in\mathcal{S}^{n+}$ be
a non-null positive semi-definite matrix. The pseudo-log-determinant
is defined by the sum of the logarithms of all the positive eigenvalues:
\[
\ld\left(A\right)=\sum_{\lambda_{i}>0}\log\left(\lambda_{i}\right)
\]
where $\lambda_{i}$ are the eigenvalues of $A$. 
\end{definition}

The interest of the PLD lies in the connection between SDD matrices
and some associated Laplacian. It is well-known that solving an SDD
system in $SDD_{n}$ can be reduced to solving a Laplacian system
of size $2n+1$, using the reduction technique introduced Gremban
in \cite{Gremban1996}. Recall that a Laplacian has all its non-diagonal
terms non-positive, the sum of each row and each column being zero.
The reduction has been simplified by Kelner et al. in \cite{kelner2013simple},
Appendix A. Using the Kelner et al.\ reduction, we can turn the computation
of a the log-determinant of a SDD system into the computation of two
PLDs of Laplacians, as shown in the next lemma. 
\begin{lemma}
\emph{Kelner et al.\ reduction for log-determinants. }Given an invertible
SDD matrix $A$, consider the Kelner decomposition $A=D_{1}+A_{p}+A_{n}+D_{2}$
where: 
\begin{itemize}
\item $A_{p}$ is the matrix that contains all the positive off-diagonal
terms of $A$ 
\item $A_{n}$is the matrix that contains all the negative off-diagonal
terms of $A$ 
\item $D_{1}$ is a diagonal matrix that verifies $D_{1}\left(i,i\right)=\sum_{j\neq i}\left|A\left(i,j\right)\right|$ 
\item $D_{2}$ is the excess diagonal matrix: $D_{2}=A-A_{p}-A_{n}-D_{1}$ 
\end{itemize}
Call $\hat{A}=D_{1}+A_{n}-A_{p}$ and $\tilde{A}=\left(\begin{array}{cc}
D_{1}+D_{2}/2+A_{n} & -D_{2}/2-A_{p}\\
-D_{2}/2-A_{p} & D_{1}+D_{2}/2+A_{n}
\end{array}\right)$. Then $\hat{A}$ and $\tilde{A}$ are both Laplacian matrices and
\[
\log\left|A\right|=\ld\left(\tilde{A}\right)-\ld\left(\hat{A}\right)
\]
\end{lemma}
\begin{proof}
The matrices $\hat{A}$ and $\tilde{A}$ are Laplacian by constructions,
and we show that the eigenvalues of $\tilde{A}$ are exactly the concatenation
of the eigenvalues of $\hat{A}$ and $A$. Call $\lambda_{i}$ an
eigenvalue of $A$ with $x$ an associated eigenvector. Then the vector
$\left(\begin{array}{c}
x\\
-x
\end{array}\right)$ is an eigenvector of $\tilde{A}$ with associated eigenvalue $\lambda$.
Similarly, call $\mu_{i}$ an eigenvalue of $\hat{A}$ with $y$ an
associated eigenvector. Then $\mu$ is an eigenvalue of $\tilde{A}$
with associated eigenvector $\left(\begin{array}{c}
y\\
y
\end{array}\right)$. Since $\tilde{A}$ is exactly of size $2n$, the set of eigenvalues
of $\tilde{A}$ is exactly the concatenation of the eigenvalues of
$\hat{A}$ and $A$. By definition of the PLD: $\ld\left(\tilde{A}\right)=\sum_{i:\lambda_{i}>0}\log\lambda_{i}+\sum_{\mu_{i}>0}\log\mu_{i}$.
Since $A$ is invertible, $\lambda_{i}>0$ for all $i$ and $\sum_{i:\lambda_{i}>0}\log\lambda_{i}=\sum_{i}\log\lambda_{i}=\log\left|A\right|$.
Finally, by definition of the PLD, we get $\sum_{\mu_{i}>0}\log\mu_{i}=\ld\left(\hat{A}\right)$. 
\end{proof}

To any Laplacian $L$ we can associate a unique positive definite
matrix $F_{L}$ (up to a permutation), and this transform preserves
eigenvalues and matrix inequalities. We call this process ``floating''
of the Laplacian, by analogy to the ``grounding'' in the electrical
sense of the SDD matrix as a Laplacian introduced by Gremban (see
\cite{Gremban1996}, Chapter 4). 
\begin{definition}
\emph{Floating a Laplacian}. Consider $L$ a Laplacian matrix. Call
$F_{L}$ the matrix formed by removing the last row and the last column
from $L$. 
\end{definition}

The following lemma shows that the Laplacian matrix overdetermines
a system, and that no information is lost by floating it. 
\begin{lemma}
\label{lem:floating-properties}Consider $Z$ a (weighted) Laplacian
matrix of a connected graph, then: 
\begin{enumerate}
\item The eigenvalues of $F_{Z}$ are the positive eigenvalues of $Z$,
and the corresponding eigenvectors for $F_{Z}$ are the same eigenvectors,
truncated by the last coefficient. 
\item $\ld\left(Z\right)=\log\left|F_{Z}\right|$ 
\item Given $Z_{1},Z_{2}$ Laplacian matrices, we have $Z_{1}\preceq Z_{2}\Rightarrow F_{Z_{1}}\preceq F_{Z_{2}}$
. 
\end{enumerate}
\end{lemma}

The proof of this lemma is straightforward, and is contained in Appendix
B.

A Laplacian matrix can be considered either for its graphical properties,
or for its algebraic properties. Recent results have shown a deep
connection between these two aspects, and they let us develop a general
framework for computing determinants: consider a Laplacian $L_{G}$
identified to its graph $G$. Using graphical properties of $L_{G}$,
we can construct a subgraph $H$ of $G$ for which the PLD is easier
to compute and that is a good approximation of $G$ in the spectral
sense. Then we can float the subgraph $H$ and apply results of section
\ref{sec:Preconditioned-log-determinants} to approximate the remainder
with high probability. More precisely: 
\begin{eqnarray*}
\ld\left(L_{G}\right) & = & \log\left|F_{L_{G}}\right|\\
 & = & \ld\left(L_{H}\right)-\log\left|F_{L_{H}}\right|+\log\left|F_{L_{G}}\right|\\
 & = & \ld\left(L_{H}\right)+\log\left|F_{L_{H}}^{-1}F_{L_{G}}\right|
\end{eqnarray*}

The first term $\ld\left(L_{H}\right)$ is usually easier to compute
by considering the graphical properties of $L_{H}$, while the remainder
$\log\left|F_{L_{H}}^{-1}F_{L_{G}}\right|$ is approximated by sampling.
Preconditioner graphs $L_{H}$ are typically efficient to factorize
using Cholesky factorization, and close enough to $G$ so that the
sampling procedure from the previous section can be applied to compute
$\log\left|F_{L_{H}}^{-1}F_{L_{G}}\right|$. We will see how to adapt
Spielman and Teng's remarkable work on \emph{ultra-sparsifiers} to
produce good preconditioners $H$ for the determinant. 


\subsection{A first preconditioner\label{sec:A-first-preconditioner}}

While the results in this section are not the main claims of this
paper, we hope they will provide some intuition, and an easier path
towards an implementation.

We present a first preconditioner that is not optimal, but that will
motivate our results for stronger preconditioners: a tree that spans
the graph $G$. Every graph has a low-stretch spanning tree, as discovered
by Alon et al. \cite{Alon1995}. The bound of Alon et al. was then
improved by Abraham et al. \cite{Abraham2008}. We restate their main
result. 
\begin{lemma}
(Lemma 9.2 from \cite{Spielman2009a}). Consider a weighted graph
$G$. There exists a spanning tree $T$ that is a subgraph of $G$
so that: 
\[
L_{T}\preceq L_{G}\preceq\kappa L_{T}
\]
with $\kappa=\tilde{\mathcal{O}}\left(m\log n\right)$. 
\label{lem:tree-st} \end{lemma}
\begin{proof}
This follows directly from \cite{Spielman2009a}. $T$ is a subgraph
of $G$ (with the same weights on the edges), so $L_{T}\preceq L_{G}$
(see \cite{Spielman2009a} for example for a proof of this fact).
Furthermore, we have $L_{G}\preceq\text{st}_{T}\left(G\right)L_{T}$.
This latter inequality is a result of Spielman et al. in \cite{Spielman2010}
that we will generalize further in Lemma \ref{lem:stretch-inequality}.
Finally, a result by \cite{Abraham2008} shows that $T$ can be chosen
such that $\text{st}_{T}\left(G\right)\leq\mathcal{O}(m\log n(\log\log n)^{3})$. 
\end{proof}

Trees enjoy a lot of convenient properties for Gaussian elimination.
The Cholesky factorization of a tree can be computed in linear time,
and furthermore this factorization has a linear number of non-zero
elements \cite{Spielman2009a}. This factorization can be expressed
as: 
\[
L_{T}=PLDL^{T}P^{T}
\]
where $P$ is a permutation matrix, $L$ is a lower-triangular matrix
with the diagonal being all ones, and $D$ a diagonal matrix in which
all the elements but the last one are positive, the last element being
$0$. These well-known facts about trees are presented in \cite{Spielman2009a}.
Once the Cholesky factorization of the tree is performed, the log-determinant
of the original graph is an immediate by-product: 
\[
\log\left|L_{T}\right|=\sum_{i=1}^{n-1}\log D_{ii}
\]
Furthermore, computing $L_{T}^{+}x$ also takes $\mathcal{O}\left(n\right)$
computations by forward-backward substitution (see \cite{duff1986direct}).
Combining Corollary \ref{cor:preconditioning} and Lemma \eqref{lem:tree-st}
gives immediately the following result. 
\begin{theorem}
\label{thm:PLD-tree}Let $G$ be a graph with $n$ vertices and $m$
edges. Its PLD can be computed up to a precision $\epsilon$ and with
high probability in time: 
\[
\tilde{O}\left(m^{2}\log n\log^{2}\left(m\right)\left(\frac{1}{\epsilon}+\frac{1}{n\epsilon^{2}}\right)\log\left(\frac{2m}{\epsilon}\right)\log\left(2/\eta\right)\right)
\]
\end{theorem}
\begin{proof}
Using Lemma \eqref{lem:tree-st}, we compute a low-stretch tree $L_{T}$
so that $L_{T}\preceq L_{G}\preceq\kappa L_{T}$ with $\kappa=\tilde{\mathcal{O}}\left(m\log n\right)$.
Using Corollary \eqref{cor:preconditioning}, approximating the PLD
with high precision requires 
\begin{eqnarray*}
\tilde{O}\left(\kappa\left(\frac{1}{\epsilon}+\frac{1}{n\epsilon^{2}}\right)\log\left(\frac{2\kappa}{\epsilon}\right)\log\left(2/\eta\right)\log^{2}\left(\kappa\right)\right)\\
=\tilde{O}\left(m\log n\left(\frac{1}{\epsilon}+\frac{1}{n\epsilon^{2}}\right)\log\left(\frac{2m}{\epsilon}\right)\log\left(2/\eta\right)\log^{2}\left(m\right)\right)
\end{eqnarray*}
inversions by the tree $T$ (done in $\mathcal{O}\left(n\right)$)
and vector products by the floated Laplacian $F_{L_{G}}$(done in
$\mathcal{O}\left(m\right)$). The overall cost is\[ \tilde{O}\left(m^{2}\log n\log^{2}\left(m\right)\left(\frac{1}{\epsilon}+\frac{1}{n\epsilon^{2}}\right)\log\left(\frac{2m}{\epsilon}\right)\log\left(2/\eta\right)\right).\]
\end{proof}
The previous result shows that the log-determinant can be computed
in roughly $\mathcal{O}\left(m^{2}\right)$ ($m$ being the number
of non-zero entries). This result may be of independent interest since
it requires relatively little machinery to compute, and it is a theoretical
improvement already for graphs with small vertex degree ($m=\mathcal{O}\left(n^{1+o\left(1\right)}\right)$)
over the Cholesky factorization of $G$ (which has complexity $\mathcal{O}\left(n^{3}\right)$
in all generality). Also, note that the PLD of the tree constructed
above provides an upper bound to the log-determinant of $G$ since
$L_{G}\preceq\kappa L_{T}$. We will see in Subsection \ref{sub:Stretch-bounds}
that we can compute a non-trivial lower bound as well. 


\subsection{Incremental sparsifiers\label{sec:Incremental-sparsifiers}}

We can do better and achieve near-linear time by using ultra-sparsifiers.
The main insight of our result is that the class preconditioners presented
by Spielman and Teng are based on incomplete Cholesky factorization,
and hence have a determinant that is relatively easy to compute, and
furthermore that they are excellent spectral preconditioners, so the
procedure \texttt{PreconditionedLogDetMonteCarlo} is efficient to
apply. We reintroduce some concepts presented in \cite{Koutis2010}
to present a self-contained result. The following paragraphs are well-known
facts about Spielman-Teng preconditioners and have been presented
in \cite{Koutis2010,Spielman2009a}.

The central idea to the Spielman-Teng preconditioner is to sample
$\mathcal{O}\left(n\right)$ edges from the graph $A$, to form a
subgraph $B$ that is close to a tree (hence it is easy to compute
some partial Cholesky factorization), yet it is close to the original
$A$ is the spectral sense ($A\preceq B\preceq\kappa A$), thanks
to the additional edges. The partial Cholesky factorization is computed
using the \texttt{GreedyElimination} algorithm presented in \cite{Koutis2010}.
In order for this section to be self-contained, we include here the
main results of Section 4 in \cite{Spielman2009a}.

Consider the Laplacian matrix $L_{B}$ of the subgraph $B$. There
exists an algorithm that computes the partial Cholesky factorization:
\[
L_{B}=PLCL^{T}P^{T}
\]
where: 
\begin{itemize}
\item $P$ is a permutation matrix 
\item $L$ is a non-singular, low triangular matrix of the form 
\[
L=\left(\begin{array}{cc}
L_{1,1} & 0\\
L_{2,1} & I_{n_{1}}
\end{array}\right)
\]
with the diagonal of $L_{1,1}$ being all ones. 
\item $C$ has the form 
\[
C=\left(\begin{array}{cc}
D_{n-n_{1}} & 0\\
0 & L_{A_{1}}
\end{array}\right)
\]
and every row and column of $L_{A_{1}}$ has at least 3 non-zero coefficients.
Furthermore, $L_{A_{1}}$ is itself Laplacian and: 
\[
\text{ld}\left(L_{G}\right)=\sum_{1}^{n-n_{1}}\log D_{ii}+\mbox{\text{ld}}\left(L_{A_{1}}\right)
\]

\end{itemize}
The exact algorithm that achieves this factorization is called \texttt{GreedyElimination}
and is presented in \cite{Koutis2010}. Using this factorization,
the PLD of the original Laplacian $L_{A}$ is: 
\begin{eqnarray}
\text{ld}\left(L_{A}\right) & = & \text{ld}\left(L_{B}\right)+\text{ld}\left(B^{+}A\right)\nonumber \\
 & = & \sum_{1}^{n-n_{1}}\log D_{ii}+\mbox{\text{ld}}\left(A_{1}\right)+\text{ld}\left(B^{+}A\right)\label{eq:chain-recursion}
\end{eqnarray}
Thus, we are left with solving a smaller problem $A_{1}$, and we
approximate the value of $\text{ld}\left(B^{+}A\right)$ using the
algorithm \texttt{SampleLogDet}. ST preconditioners are appealing
for this task: they guarantee that $A_{1}$ is substantially smaller
than $A$, so the recursion completes in $\mathcal{O}\left(\log n\right)$
steps. Furthermore, computing the vector product $B^{+}Ax$ is itself
efficient (in can be done approximated in near-linear time), so we
can apply Theorem \ref{thm:preconditioning-approx}. We formalize
the notion of chain of preconditioners by reintroducing some material
from \cite{Koutis2010}.

\begin{algorithm}
\begin{raggedright} Algorithm \textbf{UltraLogDet}($A$,$\epsilon$,$\eta$):

\end{raggedright}

\begin{raggedright} If $A$ is of a small size ($<$100), directly
compute $\text{ld}\left(A\right)$ with a dense Cholesky factorization.

\end{raggedright}

\begin{raggedright} Compute $B=$\textbf{IncrementalSparsify($A$)}

\end{raggedright}

\begin{raggedright} Compute $D,A'=$\textbf{PartialCholesky($B$)}

\end{raggedright}

\begin{raggedright} $\eta\leftarrow\min\left(\frac{\epsilon}{8\kappa^{3}\kappa\left(B\right)},\frac{1}{2\kappa}\right)$

\end{raggedright}

\begin{raggedright} $p\leftarrow8\left(\frac{1}{\epsilon}+\frac{1}{n\epsilon^{2}}\right)\log\left(\eta^{-1}\right)\log^{2}\left(\delta^{-1}\right)$

\end{raggedright}

\begin{raggedright} $l\leftarrow\delta^{-1}\log\left(\frac{2}{\epsilon\delta}\right)$

\end{raggedright}

\begin{raggedright} Compute $s=$\textbf{PreconditionedLogDetMonteCarlo($B,A,\eta,p,l$)}

\end{raggedright}

\begin{raggedright} Return $s+\log\left|D\right|+$\textbf{UltraLogDet($A'$,$\epsilon$,$\eta$)}

\end{raggedright}

\caption{Sketch of the main algorithm\label{alg:The-main-algorithm}}
\end{algorithm}

\begin{definition}
Definition 4.2 from \cite{Koutis2011}. Good preconditioning chain.
Let $d\in\mathbb{N}^{*}$, $\mathcal{C}=\left\{ A_{1}=A,B_{1},A_{2},B_{2},A_{3}\dots B_{d-1},A_{d}\right\} $
be a chain of graphs and $\mathcal{K}=\left(\kappa_{1}\cdots\kappa_{d-1}\right)\in\mathbb{R}_{+}^{d-1}$.
We say that $\left\{ \mathcal{C},\mathcal{K}\right\} $ is a good
preconditioning chain for $A$ if there exists $\mathcal{U}=\left(\mu_{1}\cdots\mu_{d}\right)\in\mathbb{N}_{+}^{d}$
so that: 
\begin{enumerate}
\item $A_{i}\preceq B_{i}\preceq\kappa_{i}A_{i}$ . 
\item $A_{i+1}$= \texttt{GreedyElimination}$\left(B_{i}\right)$ . 

\item The number of edges of $A_{i}$ is less than $\mu_{i}$. 
\item $\mu_{1}=\mu_{2}=m$ where $m$ is the number of edges of $A$. 
\item $\mu_{i}/\mu_{i+1}\geq c_{r}\left\lceil \sqrt{\kappa_{i}}\right\rceil $
for some constant $c_{r}$. 
\item $\kappa_{i+1}\leq\kappa_{i}$. 
\item $\mu_{d}$ is smaller than some fixed arbitrary constant. 
\end{enumerate}
\end{definition}
Good chains exist, as found by Koutis, Miller and Peng: 
\begin{lemma}
\label{lem:good-chain} (Lemma 4.5 from \cite{Koutis2011}) Given
a graph $A$, the algorithm \texttt{BuildChain}$\left(A,p\right)$
from \cite{Koutis2011} produces with probability $1-p$ a good preconditioning
chain $\left\{ \mathcal{C},\mathcal{K}\right\} $ such that $\kappa_{1}=\tilde{O}\left(\log^{2}n\right)$
and $\kappa_{i}=\kappa_{c}$ for all $i\geq2$ for some constant $\kappa_{c}$.
The length of the chain is $d=\mathcal{O}\left(\log n\right)$ and
the algorithm runs in expected time $\tilde{O}\left(m\log n\right).$

\end{lemma}
These chains furthermore can be used as good preconditioners for conjugate
gradient and lead to near-linear algorithms for approximate inversion
(Lemma 7.2 from \cite{Koutis2010}). This remarkable result has been
significantly strengthened in the previous years, so that SDD systems
can be considered to be solved in (expected) linear time. 
\begin{lemma}
\label{lem:linear-precond-existence}(Theorem 4.6 from \cite{Koutis2011}).
Given $A\in SDD_{n}$ with $m$ non-zero entries, $b\in\mathbb{R}^{n}$
and $\nu>0$, a vector $x$ such that $\left\Vert x-A^{+}b\right\Vert _{A}<\nu\left\Vert A^{+}b\right\Vert _{A}$
can be computed in expected time $\tilde{O}\left(m\log n\log\left(1/\nu\right)\right)$. 
\end{lemma}

It should now become clear how we can combine a good chain with the
Algorithm \texttt{PreconditionedLogDetMonteCarlo}. We start by building
a chain. The partial Cholesky factorizations at each step of the chain
provide an upper bound on $\mbox{ld}\left(A\right)$. We then refine
this upper bound by running \texttt{Preconditioned\-LogDetMonteCarlo}
at each state of the chain to approximate $\mbox{ld}\left(B_{i}^{+}A_{i}\right)$
with high probability. The complete algorithm is presented in Algorithm
\ref{alg:The-main-algorithm}. We now have all the tools required
to prove Theorem \ref{thm:ultra_main}.

\textbf{Proof of Theorem \ref{thm:ultra_main}.} First, recall that
we can consider either an SDD or its grounded Laplacian thanks to
the relation $\log |A|=\ld L_{A}$. Call $A_{1}=L_{A}$ the first element
of the chain. In this proof, all the matrices will be Laplacian from
now on. Using Lemma \ref{lem:good-chain}, consider $\mathcal{C}=\left\{ A_{1}=A,B_{1},A_{2},\dots A_{d}\right\} $
a good chain for $A$, with $d=\mathcal{O}\left(\log n\right)$. More
precisely, since $A_{i+1}$= \texttt{Greedy\-Elimination}$\left(B_{i}\right)$,
the Laplacian $B_{i}$ can be factored as: 
\[
B_{i}=P_{i}L_{i}\left(\begin{array}{cc}
D^{\left(i\right)} & 0\\
0 & A_{i+1}
\end{array}\right)L_{i}^{T}P_{i}^{T}
\]
with $P_{i}$ a permutation matrix, $L_{i}$ a lower triangular matrix
with $1$ one the diagonal and $D^{\left(i\right)}$a positive definite
diagonal matrix. The matrix $D^{\left(i\right)}$ is an immediate
by-product of running the algorithm \texttt{GreedyElimination} and
can be obtained when forming the chain $\mathcal{C}$ at no additional
cost.

From the discussion at the start of the section, it is clear that
$\ld B_{i}=\sum_{k}\log D_{k}^{\left(i\right)}+\ld A_{i+1}$. From
the discussion in Section \ref{sec:Preconditioned-log-determinants},
the log-determinant of $A$ is: 
\begin{eqnarray*}
\log |A| & = & \ld A_{1}\\
 & = & \ld B_{1}+\ld\left(B_{1}^{+}A_{1}\right)\\
 & = & \sum_{k}\log D_{k}^{\left(1\right)}+\ld A_{2}+\ld\left(B_{1}^{+}A_{1}\right)\\
 & ... & \\
 & = & \ld A_{d}+\sum_{i=1}^{d}\left(\sum_{k}\log D_{k}^{\left(i\right)}\right)+\sum_{i=1}^{d}\ld\left(B_{i}^{+}A_{i}\right)
\end{eqnarray*}
The term $\ld A_{d}$ can be estimated by dense Cholesky factorization
at cost $\mathcal{O}\left(1\right)$, and the diagonal Cholesky terms
$\sum_{k}\log D_{k}^{\left(i\right)}$ are already computed from the
chain. We are left with estimating the $d$ remainders $\text{ld}\left(B_{i}^{+}A_{i}\right)$.
By construction, $A_{i}\preceq B_{i}\preceq\kappa_{i}A_{i}$ and by
Lemma \ref{lem:linear-precond-existence}, there exists an operator
$C_{i}$ so that $\left\Vert C_{i}\left(b\right)-B_{i}^{+}b\right\Vert _{B_{i}}<\nu\left\Vert B_{i}^{+}b\right\Vert _{B_{i}}$
for all $b$ with a choice of relative precision $\nu=\frac{\epsilon}{16\kappa_{i}^{3}\kappa\left(B_{i}\right)}$.

This relative precision depends on the condition number $\kappa\left(B_{i}\right)$
of $B_{i}$. We can coarsely relate this condition number to the condition
number of $A_{1}$by noting the following: 
\begin{itemize}
\item Since $A_{i}\preceq B_{i}\preceq\kappa_{i}A_{i}$ by construction,
$\kappa\left(B_{i}\right)\leq\kappa_{i}\kappa\left(A_{i}\right)$ 
\item For diagonally dominant matrices or Laplacian matrices, the condition
number of the partial Cholesky factor is bounded by the condition
number of the original matrix. This can be seen by analyzing one update
in the Cholesky factorization. Given a partially factorized matrix
$\tilde{A}=\left(\begin{array}{ccc}
I_{p} & 0 & 0\\
0 & a & b\\
0 & b^{T} & S
\end{array}\right)$, after factorization, the next matrix is $\left(\begin{array}{cc}
I_{p+1} & 0\\
0 & S-a^{-1}bb^{T}
\end{array}\right)$. The spectrum of the Schur complement $S-a^{-1}bb^{T}$is bounded
by the spectrum of $\left(\begin{array}{cc}
a & b\\
b^{T} & S
\end{array}\right)$ (see Corollary 2.3 in \cite{Zhang2005}) and thus its condition number
is upper bounded by that of $\tilde{A}$. 
\end{itemize}
As a consequence, we have for all $i$: $\kappa\left(A_{i+1}\right)\leq\kappa\left(B_{i}\right)\leq\kappa_{i}\kappa\left(A_{i}\right)\leq\prod_{j=1}^{i}\kappa_{j}\kappa\left(A_{1}\right) \\*=\tilde{O}\left(\kappa_{1}\kappa_{c}^{i-1}\kappa\left(A\right)\right)$
with $\kappa_{c}$ the constant introduced in Lemma \ref{lem:good-chain}.
This coarse analysis gives us the bound: 
\[
\kappa\left(B_{i}\right)\leq\tilde{O}\left(\kappa_{c}^{\log n}\log^{2}n~\kappa\left(A\right)\right)=\tilde{O}\left(n^{\log\kappa_{c}}\log^{2}n\ \kappa\left(A\right)\right).
\]
Consider the relative precision $\tilde{\nu}=\tilde{\mathcal{O}}\left(n^{-\log\kappa_{c}}\log^{-8}n\frac{\epsilon}{\kappa\left(A\right)}\right)$
so that $\tilde{\nu}\leq\nu_{i}$ for all $i$. Constructing the operator
$C_{i}$ is a byproduct of forming the chain \emph{$\mathcal{C}$.
}By Theorem \ref{thm:det-sampling-theorem}, each remainder $\text{ld}\left(B_{i}^{+}A_{i}\right)$
can be approximated to precision $\epsilon$ with probability at least
$1-\eta$ using Algorithm \ref{alg:SampleLogDet}. Furthermore, this
algorithm works in expected time 
\begin{eqnarray*}
 &  & \tilde{O}\left(m\log n\log\left(1/\tilde{\nu}\right)\kappa_{1}\left(\frac{1}{\epsilon}+\frac{1}{n\epsilon^{2}}\right)\log\left(\frac{n\kappa_{1}}{\tilde{\nu}}\right)\log^{2}\left(\kappa_{1}\right)\log\left(\eta^{-1}\right)\right)\\
 &  & =\tilde{O}\left(m\log^{3}n\left(\frac{1}{\epsilon}+\frac{1}{n\epsilon^{2}}\right)\log^{2}\left(\frac{n\kappa\left(A\right)}{\epsilon}\right)\log\left(\eta^{-1}\right)\right)
\end{eqnarray*}
By a union bound, the result also holds on the sum of all the $\log n$
approximations of the remainders. We can simplify this bound a little
by assuming that $\epsilon\geq n^{-1}$, which then becomes $\tilde{O}\left(m\epsilon^{-1}\log^{3}n\log^{2}\left(\frac{n\kappa\left(A\right)}{\epsilon}\right)\log\left(\eta^{-1}\right)\right)$. 


\subsection{Stretch bounds on preconditioners\label{sub:Stretch-bounds}}

How good is the estimate provided by the preconditioner? Intuitively,
this depends on how well the preconditioner $L_{H}$ approximates
the graph $L_{G}$. This notion of quality of approximation can be
formalized by the notion of \emph{stretch}. This section presents
a deterministic bound on the PLD of $L_{G}$ based on the PLD of $L_{H}$
and the stretch of $G$ relative to $H$. This may be useful in practice
as it gives a (tight) interval for the PLD before performing any Monte-Carlo
estimation of the residual.

The stretch of a graph is usually defined with respect to a (spanning)
tree. In our analysis, it is convenient and straightforward to generalize
this definition to arbitrary graphs. To our knowledge, this straightforward
extension is not considered in the literature, so we feel compelled
to properly introduce it. 
\begin{definition}
\emph{Generalized stretch}.\label{Generalized-stretch} Consider $\mathcal{V}$
a set of vertices, $G=\left(\mathcal{V},\,\mathcal{E}_{G}\right),\, H=\left(\mathcal{V},\,\mathcal{E}_{H}\right)$
connected graphs over the same set of vertices, and $L_{G}$, $L_{H}$
their respective Laplacians. The stretch of $G$ with respect to $H$
is the sum of the effective resistances of each edge of graph $G$
with respect to graph $H$, 
\[
\text{st}_{H}\left(G\right)=\sum_{\left(u,v\right)\in\mathcal{E}_{G}}L_{G}\left(u,v\right)\left(\mathcal{X}_{u}-\mathcal{X}_{v}\right)^{T}L_{H}^{+}\left(\mathcal{X}_{u}-\mathcal{X}_{v}\right)
\]
with $\mathcal{X}_{u}\in\mathbb{R}^{n}$ the unit vector that is $1$
at position $u$, and zero otherwise. 
\end{definition}

If the graph $H$ is a tree, this is a standard definition of stretch,
because the effective resistance $\left(\mathcal{X}_{u}-\mathcal{X}_{v}\right)^{T}L_{H}^{+}\left(\mathcal{X}_{u}-\mathcal{X}_{v}\right)$
between vertices $u$ and $v$ is the sum of all resistances over
the unique path between $u$ and $v$ (see Lemma 2.4 in \cite{Spielman2009b}).
Furthermore, the arguments to prove Theorem 2.1 in \cite{Spielman2009b}
carry over to our definition of stretch. For the sake of completeness,
we include this result: 
\begin{lemma}
\label{lem:stretch-trace}(Straightforward generalization of Theorem
2.1 in \cite{Spielman2009b}) Let $G=\left(\mathcal{V},\,\mathcal{E}_{G}\right),\, H=\left(\mathcal{V},\,\mathcal{E}_{H}\right)$
be connected graphs over the same set of vertices, and $L_{G}$, $L_{H}$
their respective Laplacians. Then: 
\[
\text{st}_{H}\left(G\right)=\text{Tr}\left(L_{H}^{+}L_{G}\right)
\]
with $L_{H}^{+}$the pseudo-inverse of $L_{H}$.\end{lemma}
\begin{proof}
We denote $E\left(u,v\right)$ the Laplacian unit matrix that is $1$
in position $u,v$: $E\left(u,v\right)=\left(\mathcal{X}_{u}-\mathcal{X}_{v}\right)\left(\mathcal{X}_{u}-\mathcal{X}_{v}\right)^{T}$.
This is the same arguments as the original proof: 
\begin{eqnarray*}
\text{Tr}\left(L_{H}^{+}L_{G}\right) & = & \sum_{\left(u,v\right)\in\mathcal{E}_{G}}L_{G}\left(u,v\right)\text{Tr}\left(E\left(u,v\right)L_{H}^{+}\right)\\
 & = & \sum_{\left(u,v\right)\in\mathcal{E}_{G}}L_{G}\left(u,v\right)\text{Tr}\left(\left(\mathcal{X}_{u}-\mathcal{X}_{v}\right)\left(\mathcal{X}_{u}-\mathcal{X}_{v}\right)^{T}L_{H}^{+}\right)\\
 & = & \sum_{\left(u,v\right)\in\mathcal{E}_{G}}L_{G}\left(u,v\right)\left(\mathcal{X}_{u}-\mathcal{X}_{v}\right)^{T}L_{H}^{+}\left(\mathcal{X}_{u}-\mathcal{X}_{v}\right)\\
 & = & \text{st}_{H}\left(G\right)
\end{eqnarray*}

\end{proof}
A consequence is $\text{st}_{H}\left(G\right)\geq\text{Card}\left(\mathcal{E}_{G}\right)\geq n-1$
for connected $G$ and $H$ with $L_{G}\succeq L_{H}$, and that for
any connected graph $G$, $\text{st}_{G}\left(G\right)=n-1$. Scaling
and matrix inequalities carry over with the stretch as well. Given
$A,B,C$ connected graphs, and $\alpha,\beta>0$: 
\begin{align*}
\text{st}_{\alpha A}\left(\beta B\right) & =\alpha^{-1}\beta\text{st}_{A}\left(B\right)\\
L_{A}\preceq L_{B} & \Rightarrow\text{st}_{A}\left(C\right)\geq\text{st}_{B}\left(C\right)\\
L_{A}\preceq L_{B} & \Rightarrow\text{st}_{C}\left(A\right)\leq\text{st}_{C}\left(B\right)
\end{align*}

\begin{lemma}
For any connected graph $G$, $\text{st}_{G}\left(G\right)=n-1$.\end{lemma}
\begin{proof}
Consider the diagonalization of $L_{G}$: $L_{G}=P\Delta P^{T}$ with
$P\in\mathbb{R}^{n\times n-1}$ and $\Delta=\text{diag}\left(\lambda_{1},\cdots,\lambda_{n-1}\right)$.
Then 
\[
\text{st}_{G}\left(G\right)=\text{Tr}\left(P\Delta P^{T}P\Delta^{-1}P^{T}\right)=\text{Tr}\left(I_{n-1}\right)=n-1
\]

\end{proof}
A number of properties of the stretch extend to general graphs using
the generalized stretch. In particular, the stretch inequality (Lemma
8.2 in \cite{Spielman2009a}) can be generalized to arbitrary graphs
(instead of spanning trees). 
\begin{lemma}
\label{lem:stretch-inequality}Let $G=\left(\mathcal{V},\,\mathcal{E}_{G}\right),\, H=\left(\mathcal{V},\,\mathcal{E}_{H}\right)$
be connected graphs over the same set of vertices, and $L_{G}$, $L_{H}$
their respective Laplacians. Then: 
\[
L_{G}\preceq\text{st}_{H}\left(G\right)L_{H}
\]
\end{lemma}
\begin{proof}
The proof is very similar to that of Lemma 8.2 in \cite{Spielman2009b},
except that the invocation of Lemma 8.1 is replaced by invoking Lemma
\ref{lem:simple-inequality} in Appendix B. The Laplacian $G$ can
be written as a linear combination of edge Laplacian matrices: 
\[
L_{G}=\sum_{e\in\mathcal{E}_{G}}\omega_{e}L\left(e\right)=\sum_{\left(u,v\right)\in\mathcal{E}_{G}}\omega_{\left(u,v\right)}\left(\mathcal{X}_{u}-\mathcal{X}_{v}\right)\left(\mathcal{X}_{u}-\mathcal{X}_{v}\right)^{T}
\]
and a positivity result on the Schur complement gives
\[
\left(\mathcal{X}_{u}-\mathcal{X}_{v}\right)\left(\mathcal{X}_{u}-\mathcal{X}_{v}\right)^{T}\preceq\Big(\left(\mathcal{X}_{u}-\mathcal{X}_{v}\right)^{T}L_{H}^{+}\left(\mathcal{X}_{u}-\mathcal{X}_{v}\right)\Big)L_{H}
\]
By summing all the edge inequalities, we get: 
\begin{eqnarray*}
L_{G} & \preceq & \sum_{\left(u,v\right)\in\mathcal{E}_{G}}\omega_{\left(u,v\right)}\left(\mathcal{X}_{u}-\mathcal{X}_{v}\right)^{T}L_{H}^{+}\left(\mathcal{X}_{u}-\mathcal{X}_{v}\right)L_{H}\\
 & \preceq & \text{st}_{H}\left(G\right)L_{H}
\end{eqnarray*}

\end{proof}
This bound is remarkable as it relates any pair of (connected) graphs,
as opposed to spanning trees or subgraphs. An approximation of the
generalized stretch can be quickly computed using a construct detailed
in \cite{Spielman2009}, as we will see below. We now introduce the
main result of this section: a bound on the PLD of $L_{G}$ using
the PLD of $L_{H}$ and the stretch. 
\begin{theorem}
\label{thm:stretch-pld-bounds}Let $G=\left(\mathcal{V},\,\mathcal{E}_{G}\right),\, H=\left(\mathcal{V},\,\mathcal{E}_{H}\right)$
be connected graphs over the same set of vertices, and $L_{G}$, $L_{H}$
their respective Laplacians. Assuming $L_{H}\preceq L_{G}$, then:
\begin{equation}
\text{ld}\left(L_{H}\right)+\log\left(\text{st}_{H}\left(G\right)-n+2\right)\leq\text{ld}\left(L_{G}\right)\leq\text{ld}\left(L_{H}\right)+\left(n-1\right)\log\left(\frac{\text{st}_{H}\left(G\right)}{n-1}\right)\label{eq:encadrement-1}
\end{equation}

This bound is tight.\end{theorem}
\begin{proof}
This is an application of Jensen's inequality on $\text{ld}\left(L_{H}^{+}L_{G}\right)$.
We have $\text{ld}(L_{G})=\text{ld}(L_{H})+\text{ld}\left(L_{H}^{+}G\right)$
and $\text{ld}\left(L_{H}^{+}G\right)=\text{ld}\left(\sqrt{L_{H}}^{+}L_{G}\sqrt{L_{H}}^{+}\right)$
with $\sqrt{T}$ the matrix square root of $T$. From Lemma \ref{lem:Jensen-inequality-matrix-logarithm-1},
we have the following inequality: 
\begin{eqnarray*}
\text{ld}\left(\sqrt{L_{H}}^{+}L_{G}\sqrt{L_{H}}^{+}\right) & \leq & \left(n-1\right)\log\left(\frac{\text{Tr}\left(\sqrt{L_{H}}^{+}L_{G}\sqrt{L_{H}}^{+}\right)}{n-1}\right)\\
 & = & \left(n-1\right)\log\left(\frac{\text{Tr}\left(L_{H}^{+}L_{G}\right)}{n-1}\right)\\
 & = & \left(n-1\right)\log\left(\frac{\text{st}_{H}\left(G\right)}{n-1}\right)
\end{eqnarray*}

The latter equality is an application of Lemma \ref{lem:stretch-trace}.

The lower bound is slightly more involved. Call $\lambda_{i}$ the
positive eigenvalues of $\sqrt{L_{H}}^{+}L_{G}\sqrt{L_{H}}^{+}$ and
$\sigma=\text{st}_{H}\left(G\right)$. We have $1\leq\lambda_{i}$
from the assumption $L_{H}\preceq L_{G}$. By definition: $\text{ld}\left(L_{H}^{+}L_{G}\right)=\sum_{i}\log\lambda_{i}$.
Furthermore, we know from Lemma \ref{lem:stretch-trace} that $\sum_{i}\lambda_{i}=\sigma$.
The upper and lower bounds on $\lambda_{i}$ give: 
\begin{align*}
\text{ld}\left(L_{H}^{+}L_{G}\right)\geq & \min\,\sum_{i}\log\lambda_{i}\\
 & \text{s.t.}\,\lambda_{i}\geq1,\,\sum_{i}\lambda_{i}=\sigma
\end{align*}
Since there are precisely $n-1$ positive eigenvalues $\lambda_{i}$,
one can show that the minimization problem above has a unique minimum
which is $\log\left(\sigma-n+2\right)$.

To see that, consider the equivalent problem of minimizing $\sum_{i}\log\left(1+u_{i}\right)$
under the constraints $\sum_{i}u_{i}=\sigma-\left(n-1\right)$ and
$u_{i}\geq0$. Note that: 
\[
\sum_{i}\log\left(1+u_{i}\right)=\log\left(\prod_{i}\left[1+u_{i}\right]\right)=\log\left(1+\sum_{i}u_{i}+\text{Poly}\left(u\right)\right)
\]
with $\text{Poly}\left(u\right)\geq0$ for all $u_{i}\geq0$, so we
get: $\sum_{i}\log\left(1+u_{i}\right)\geq\log\left(1+\sum_{i}u_{i}\right)$
and this inequality is tight for $u_{1}=\sigma-\left(n-1\right)$
and $u_{i\geq2}=0$. Thus the vector $\lambda^{*}=\left(\sigma-n+2,\,1\cdots1\right)^{T}$
is (a) a solution to the minimization problem above, and (b) the objective
value of any feasible vector $\lambda$ is higher or equal. Thus,
this is the solution (unique up to a permutation). Hence we have $\text{ld}\left(L_{H}^{+}L_{G}\right)\geq\sum_{i}\log\lambda_{i}^{*}=\log\left(\sigma-n+2\right)$.

Finally, note that if $H=G$, then $\text{st}_{H}\left(G\right)=n-1$,
which gives an equality. 
\end{proof}
Note that Lemma \ref{lem:stretch-inequality} gives us $L_{H}\preceq L_{G}\preceq\text{st}_{H}\left(G\right)L_{H}$
which implies $\text{ld}\left(L_{H}\right)\leq\text{ld}\left(L_{G}\right)\leq\text{ld}\left(L_{H}\right)+n\log\text{st}_{H}\left(G\right)$.
The inequalities in Theorem \ref{thm:stretch-pld-bounds} are stronger.
Interestingly, it does not make assumption on the topology of the
graphs (such as $L_{H}$ being a subset of $L_{G}$). Research on
conditioners has focused so far on low-stretch approximations that
are subgraphs of the original graph. It remains to be seen if some
better preconditioners can be found with stretches in $\mathcal{O}\left(n\right)$
by considering more general graphs. In this case, the machinery developed
in Section 3 would not be necessary.

From a practical perspective, the stretch can be calculated also in near-linear time with
respect to the number of non-zero entries. 
\begin{lemma}
\label{lem:stretch-approx}Let $G=\left(\mathcal{V},\,\mathcal{E}_{G}\right),\, H=\left(\mathcal{V},\,\mathcal{E}_{H}\right)$
be connected graphs over the same set of vertices, and $L_{G}$, $L_{H}$
their respective Laplacians. Call $r=\max_{e}L_{H}\left(e\right)/\min_{e}L_{H}\left(e\right)$.
Given $\epsilon>0$, there exists an algorithm that returns a scalar
$y$ so that: 
\[
\left(1-\epsilon\right)\text{st}_{H}\left(G\right)\leq y\leq\left(1+\epsilon\right)\text{st}_{H}\left(G\right)
\]
with high probability and in expected time $\tilde{\mathcal{O}}\left(m\epsilon^{-2}\log\left(rn\right)\right)$.\end{lemma}
\begin{proof}
This is a straightforward consequence of Theorem $2$ in \cite{Spielman2009}.
Once the effective resistance of an edge can be approximated in time
$\mathcal{O}\left(\log n/\epsilon^{2}\right)$, we can sum it and
weight it by the conductance in $G$ for each edge. 
\end{proof}

\subsection{Fast inexact estimates}

The bound presented in Equation \ref{eq:encadrement-1} has some interesting
consequences if one is interested only in a rough estimate of the
log-determinant: if $\epsilon=\mathcal{O}\left(1\right)$, it is possible
to approximate the log-determinant in expected time $\tilde{O}\left(m+n\log^{3}n\right)$.
We will make use of this sparsification result from Spielman and Srivastava
\cite{Spielman2009}: 
\begin{lemma}
\label{lem:sriva-sparsification}(Theorem 12 in \cite{Spielman2009}).
Given a Laplacian $L_{G}$ with $m$ edges, there is an expected $\tilde{O}\left(m/\epsilon^{2}\right)$
algorithm that produces a graph $L_{H}$ with $\mathcal{O}\left(n\log n/\epsilon^{2}\right)$
edges that satisfies $\left(1-\epsilon\right)L_{G}\preceq L_{H}\preceq\left(1+\epsilon\right)L_{G}$. 
\end{lemma}

An immediate consequence is that given any graph, we can find a graph
with a near-optimal stretch (up to an $\epsilon$ factor) and $\mathcal{O}\left(n\log n/\epsilon^{2}\right)$
edges. 
\begin{lemma}
\label{lem:low-stretch-bounding}Given a Laplacian $L_{G}$ with $m$
edges, there is an expected $\tilde{O}\left(m/\epsilon^{2}\right)$
algorithm that produces a graph $L_{H}$ with $\mathcal{O}\left(n\log n/\epsilon^{2}\right)$
edges that satisfies $\left(n-1\right)\leq\text{st}_{H}\left(G\right)\preceq\frac{1+\epsilon}{1-\epsilon}\left(n-1\right)$.\end{lemma}
\begin{proof}
Consider a graph $H$ produced by Lemma \ref{lem:sriva-sparsification},
which verifies $\left(1-\epsilon\right)L_{G}\preceq L_{H}\preceq\left(1+\epsilon\right)L_{G}$.
Using the stretch over this matrix inequality, this implies: 
\[
\text{st}_{\left(1+\epsilon\right)G}\left(G\right)\leq\text{st}_{H}\left(G\right)\leq\text{st}_{\left(1-\epsilon\right)G}\left(G\right)
\]
which is equivalent to: 
\[
\left(1+\epsilon\right)^{-1}\text{st}_{G}\left(G\right)\leq\text{st}_{H}\left(G\right)\leq\left(1-\epsilon\right)^{-1}\text{st}_{G}\left(G\right)
\]
and the stretch of a connected graph with respect to itself is $n-1$.
By rescaling $H$ to $\left(1+\epsilon\right)^{-1}H$, we get: 
\[
n-1\leq\text{st}_{H}\left(G\right)\leq\frac{1+\epsilon}{1-\epsilon}\left(n-1\right)
\]

\end{proof}

Here is the main result of this section: 
\begin{proposition}
There exists an algorithm that on input $A\in SDD_{n}$, returns an
approximation $n^{-1}\log\left|A\right|$ with precision $1/2$ in
expected time \linebreak $\tilde{O}\left(m+n\log^{3}n\log^{2}\kappa(A)\right)$
with $\kappa(A)$ the condition number of $A$.\end{proposition}
\begin{proof}
Given $L_{A}$, compute $H$ from Lemma \ref{lem:low-stretch-bounding}
using $\epsilon=1/16$ so that $\left(n-1\right)\leq\text{st}_{H}\left(G\right)\preceq\left(1+1/8\right)\left(n-1\right)$.
Then, using Theorem \ref{thm:stretch-pld-bounds}, this leads to the
bound: 
\[
\text{\text{ld}}\left(H\right)\leq\log\left|A\right|\leq\text{ld}\left(H\right)+\frac{n-1}{4}
\]
since $H$ has $\mathcal{O}\left(n\log n\right)$ edges by construction,
we can use Theorem \ref{thm:ultra_main} to compute a $1/4$- approximation
of $\text{ld}\left(H\right)$ in expected time $\tilde{O}\left(n\log^{3}n\log^{2}\left(\kappa(H)\right)\right)$.
By construction $\kappa(H)\leq\frac{1+1/16}{1-1/16}\kappa(A)$, hence
the result. 
\end{proof}

It would be interesting to see if this technique could be developed
to handle arbitrary precision as well.

\section*{Comments}

Since the bulk of the computations are performed in estimating the
residue PLD, it would be interesting to see if this could be bypassed
using better bounds based on the stretch.

Also, even if this algorithm presents a linear bound, it requires
a fairly advanced machinery (ST solvers) that may limit its practicality.
Some heuristic implementation, for example based on algebraic multi-grid
methods, could be a first step in this direction.

The authors are much indebted to Satish Rao and James Demmel for suggesting
the original idea, and to Benjamin Recht for helpful comments on the
draft of this article.

\section*{Appendix A: Proofs of Section 2}

\subsection{Proof of Theorem \ref{thm:det-sampling-theorem}}

\label{sub:det-sampling-proof} 
\begin{proof}
The proof of this theorem follows the proof of the Main Theorem in
\cite{Barry1999} with some slight modifications. Using triangular
inequality: 
\[
\left|y-\hat{y}_{p,l}\right|\leq\left|\mathbb{E}\left[\hat{y}_{p,l}\right]-\hat{y}_{p,l}\right|+\left|y-\mathbb{E}\left[\hat{y}_{p,l}\right]\right|
\]

Since $S$ is upper-bounded by $\left(1-\delta\right)I$, we have
for all $k\in\mathbb{N}$: 
\[
\left|\mbox{Tr}\left(S^{k}\right)\right|\leq n\left(1-\delta\right)^{k}
\]
We have $\mathbb{E}\left[\hat{y}_{p,l}\right]=-\sum_{i=1}^{l}i^{-1}S^{i}$
and $y=-\sum_{i=1}^{\infty}i^{-1}S^{i}$. Using again triangle inequality,
we can bound the error with respect to the expected value: 
\begin{eqnarray*}
\left|y-\mathbb{E}\left[\hat{y}_{p,l}\right]\right| & = & n^{-1}\left|\sum_{i=l+1}^{\infty}\frac{1}{i}\mbox{Tr}\left(S^{k}\right)\right|\\
 & \leq & n^{-1}\sum_{i=l+1}^{\infty}\frac{1}{i}\left|\mbox{Tr}\left(S^{k}\right)\right|\\
 & \leq & \frac{1}{n\left(l+1\right)}\sum_{i=l+1}^{\infty}\left|\mbox{Tr}\left(S^{k}\right)\right|\\
 & \leq & \frac{1}{l+1}\sum_{i=l+1}^{\infty}\left(1-\delta\right)^{k}\\
 & \leq & \frac{1}{l+1}\frac{\left(1-\delta\right)^{l+1}}{\delta}\\
 & \leq & \frac{\left(1-\delta\right)^{l+1}}{\delta}
\end{eqnarray*}

And since $\delta\le-\log\left(1-\delta\right)$, for a choice of
$l\geq\delta^{-1}\log\left(\frac{2}{\epsilon\delta}\right)$, the
latter part is less than $\epsilon/2$. We now bound the first part
using Lemma \ref{lem:bernstein-trace}. Call $H$ the truncated series:
\[
H=-\sum_{i=1}^{m}\frac{1}{i}S^{i}
\]
This truncated series is upper-bounded by $0$ ($H$ is negative,
semi-definite). The lowest eigenvalue of the truncated series can
be lower-bounded in terms of $\delta$: 
\[
H=-\sum_{i=1}^{m}\frac{1}{i}S^{i}\succeq-\sum_{i=1}^{m}\frac{1}{i}\left(1-\delta\right)^{i}I\succeq-\sum_{i=1}^{+\infty}\frac{1}{i}\left(1-\delta\right)^{i}I=\left(\log\delta\right)I
\]
We can now invoke Lemma \ref{lem:bernstein-trace} to conclude: 
\[
\mathbb{P}\left[\left|\frac{1}{p}\sum_{i=1}^{p}\left(\mathbf{u}_{i}^{T}\mathbf{u}_{i}\right)^{-1}\mathbf{u}_{i}^{T}H\mathbf{u}_{i}-n^{-1}\mbox{Tr}\left(H\right)\right|\geq\frac{\epsilon}{2}\right]\leq2\exp\left(-\frac{p\epsilon^{2}}{16\frac{\left(\log\left(1/\delta\right)\right)^{2}}{n}+4\frac{\log\left(1/\delta\right)\epsilon}{3}}\right)
\]
Thus, any choice of 
\[
p\geq16\left(\frac{1}{\epsilon}+\frac{1}{n\epsilon^{2}}\right)\log\left(2/\eta\right)\log^{2}\left(\delta^{-1}\right)\geq\log\left(2/\eta\right)\epsilon^{-2}\left(16\frac{\left(\log\left(1/\delta\right)\right)^{2}}{n}+\frac{4}{3}\epsilon\log\left(\delta^{-1}\right)\right)
\]
satisfies the inequality: $2\exp\left(-\frac{p\epsilon^{2}}{16n^{-1}\left(\log\left(1/\delta\right)\right)^{2}+4\log\left(1/\delta\right)\epsilon/3}\right)\leq\eta$. 
\end{proof}

\subsection{Proof of Corollary \ref{cor:preconditioning}}
\begin{proof}
We introduce some notations that will prove useful for the rest of
the article: 
\[
H=I-B^{-1}A
\]
\[
S=I-B^{-1/2}AB^{-1/2}
\]
with $B^{-1/2}$ the inverse of the square root%
\footnote{Given a real PSD matrix $X$, which can be diagonalized: $X=Q\Delta Q^{T}$
with $\Delta$ diagonal, and $\Delta_{ii}\geq0$. Call $Y=Q\sqrt{\Delta}Q^{T}$
the square root of $X$, then $Y^{2}=X$.%
} of the positive-definite matrix $B$. The inequality (\ref{eq:A-B-bounds})
is equivalent to $\kappa^{-1}B\preceq A\preceq B$, or also: 
\[
\left(1-\kappa^{-1}\right)I\succeq I-B^{-1/2}AB^{-1/2}\succeq0
\]
\begin{equation}
\left(1-\kappa^{-1}\right)I\succeq S\succeq0\label{eq:s-encadrement}
\end{equation}

The matrix $S$ is a contraction, and its spectral radius is determined
by $\kappa$. Furthermore, computing the determinant of $B^{-1}A$
is equivalent to computing the determinant of $I-S$: 
\begin{eqnarray*}
\log\left|I-S\right| & = & \log\left|B^{-1/2}AB^{-1/2}\right|\\
 & = & \log\left|A\right|-\log\left|B\right|\\
 & = & \log\left|B^{-1}A\right|\\
 & = & \log\left|I-H\right|
\end{eqnarray*}

and invoking Theorem \ref{thm:det-sampling-theorem} gives us bounds
on the number of calls to matrix-vector multiplies with respect to
$S$. It would seem at this point that computing the inverse square
root of $B$ is required, undermining our effort. However, we can
reorganize the terms in the series expansion to yield only full inverses
of $B$. Indeed, given $l\in\mathbb{N}^{*}$, consider the truncated
series: 
\begin{eqnarray*}
y_{l} & = & -\mbox{Tr}\left(\sum_{i=1}^{l}\frac{1}{i}S^{i}\right)\\
 & = & -\sum_{i=1}^{l}\frac{1}{i}\mbox{Tr}\left(S^{i}\right)\\
 & = & -\sum_{i=1}^{l}\frac{1}{i}\mbox{Tr}\left(\sum_{j}\left(\begin{array}{c}
j\\
i-j
\end{array}\right)\left(-1\right)^{j}\left(B^{-1/2}AB^{-1/2}\right)^{j}\right)\\
 & = & -\sum_{i=1}^{l}\frac{1}{i}\sum_{j}\left(\begin{array}{c}
j\\
i-j
\end{array}\right)\left(-1\right)^{j}\mbox{Tr}\left(\left(B^{-1/2}AB^{-1/2}\right)^{j}\right)\\
 & = & -\sum_{i=1}^{l}\frac{1}{i}\sum_{j}\left(\begin{array}{c}
j\\
i-j
\end{array}\right)\left(-1\right)^{j}\mbox{Tr}\left(\left(B^{-1}A\right)^{j}\right)\\
 & = & -\sum_{i=1}^{l}\frac{1}{i}\mbox{Tr}\left(\sum_{j}\left(\begin{array}{c}
j\\
i-j
\end{array}\right)\left(-1\right)^{j}\left(B^{-1}A\right)^{j}\right)\\
 & = & -\sum_{i=1}^{l}\frac{1}{i}\mbox{Tr}\left(H^{i}\right)
\end{eqnarray*}
Hence, the practical computation of the latter sum can be done on
$A^{-1}B$. To conclude, if we compute $p=16\left(\frac{1}{\epsilon}+\frac{1}{n\epsilon^{2}}\right)\log\left(2/\eta\right)\log^{2}\left(\kappa\right)$
truncated chains of length $l=\kappa\log\left(\frac{2\kappa}{\epsilon}\right)$,
we get our result. This requires $lp$ multiplications by $A$ and
inversions by $B$. 
\end{proof}

\subsection{Proof of Theorem \ref{thm:preconditioning-approx}}

We prove here the main result of Section \ref{sec:Preconditioned-log-determinants}.
In the following, $A$ and $B$ are positive-definite matrices in
$\mathcal{S}_{n}$, and $B$ is a $\kappa-$approximation of $A$
($A\preceq B\preceq\kappa A$). The following notations will prove
useful:

\begin{equation}
S=I-B^{-1/2}AB^{-1/2}\label{eq:S-def}
\end{equation}

\begin{equation}
R=I-B^{-1}A\label{eq:R-def}
\end{equation}

\[
\varphi=\kappa^{-1}
\]

Recall the definition of the matrix norm. Given $M\in\mathcal{S}_{n}^{+}$,
$\left\Vert M\right\Vert _{B}=\max_{x\neq0}\sqrt{\frac{x^{T}Mx}{x^{T}Bx}}$ 
\begin{lemma}
\label{lem:S-R-contractions} $S$ and $R$ are contractions for the
Euclidian and $B-$norms: 
\begin{eqnarray*}
\left\Vert S\right\Vert  & \leq & 1-\varphi\\
\left\Vert R\right\Vert  & \leq & 1-\varphi\\
\left\Vert R\right\Vert _{B} & \leq & \left(1-\varphi\right)^{2}
\end{eqnarray*}
\end{lemma}
\begin{proof}
Recall the definition of the matrix norm: $\left\Vert S\right\Vert =\max_{x^{T}x\leq1}\sqrt{x^{T}Sx}$.
Since we know from Equation (\ref{eq:s-encadrement}) that $S\preceq\left(1-\varphi\right)I$,
we get the first inequality.

The second inequality is a consequence of Proposition 3.3 from \cite{Spielman2009a}:
$A$ and $B$ have the same nullspace and we have the linear matrix
inequality $A\preceq B\preceq\kappa A$, which implies that the eigenvalues
of $B^{-1}A$ lie between $\kappa^{-1}=\varphi$ and $1$. This implies
that the eigenvalues of $I-B^{-1}A$ are between $0$ and $1-\varphi$.

Recall the definition of the matrix norm induced by the $B$-norm
over $\mathbb{R}^{n}$: 
\begin{eqnarray*}
\left\Vert R\right\Vert _{B} & = & \max_{x\neq0}\frac{\left\Vert Rx\right\Vert _{B}}{\left\Vert x\right\Vert _{B}}\\
 & = & \max_{\left\Vert x\right\Vert _{B}^{2}\leq1}\sqrt{x^{T}R^{T}BRx}\\
 & = & \max_{x^{T}Bx\leq1}\sqrt{x^{T}R^{T}BRx}\\
 & = & \max_{y^{T}y\leq1}\sqrt{y^{T}B^{-1/2}R^{T}BRB^{-1/2}y}
\end{eqnarray*}
and the latter expression simplifies: 
\begin{eqnarray*}
B^{-1/2}R^{T}BRB^{-1/2} & = & B^{-1/2}\left(I-AB^{-1}\right)B\left(I-B^{-1}A\right)B^{-1/2}\\
 & = & \left(I-B^{-1/2}AB^{-1/2}\right)\left(I-B^{-1/2}AB^{-1/2}\right)\\
 & = & S^{2}
\end{eqnarray*}
so we get: 
\[
\left\Vert R\right\Vert _{B}=\left\Vert S^{2}\right\Vert \leq\left\Vert S\right\Vert ^{2}\leq\left(1-\varphi\right)^{2}
\]

\end{proof}
The approximation of the log-determinant is performed by computing
sequences of power series $\left(R^{k}x\right)_{k}$. These chains
are computed approximately by repeated applications of the $R$ operator
on the previous element of the chain, starting from a random variable
$x_{0}$. We formalize the notion of an approximate chain. 
\begin{definition}
\emph{Approximate power sequence. }Given a linear operator $H$, a
start point $x^{\left(0\right)}\in\mathbb{R}^{n}$, and a positive-definite
matrix $D$, we define an $\epsilon-$approximate power sequence as
a sequence that does not deviate too much from the power sequence:
\[
\left\Vert x^{\left(k+1\right)}-Hx^{\left(k\right)}\right\Vert _{D}\leq\epsilon\left\Vert Hx^{\left(k\right)}\right\Vert _{D}
\]

\end{definition}
We now prove the following result that is quite intuitive: if the
operator $H$ is a contraction and if the relative error $\epsilon$
is not too great, the sum of all the errors on the chain is bounded. 
\begin{lemma}
\label{lem:nu-sequence-bound}Let $H$ be a linear operator and $D$
a norm over the space of that linear operator. Assume that the operator
$H$ is a contraction under this norm ($\left\Vert H\right\Vert _{D}<1$)
and consider $\rho\in\left(0,1\right)$ so that $\left\Vert H\right\Vert _{D}\leq\left(1-\rho\right)^{2}$.
Consider $\left(x^{\left(k\right)}\right)_{k}$ a $\nu-$approximate
power sequence for the operator $H$ and the norm $D$. If $\rho\leq1/2$
and $\nu\leq\rho/2$, the total error is bounded: 
\[
\sum_{k=0}^{\infty}\left\Vert x^{\left(k\right)}-H^{k}x^{\left(0\right)}\right\Vert _{D}\leq4\rho^{-2}\nu\left\Vert x^{\left(0\right)}\right\Vert _{D}
\]
\end{lemma}
\begin{proof}
Call $\omega_{k}=\left\Vert x^{\left(k\right)}-H^{k}x^{\left(0\right)}\right\Vert _{D}$
and $\theta_{k}=\left\Vert Hx^{\left(k\right)}\right\Vert _{D}$.
We are going to bound the rate of convergence of these two series.
We have first using triangular inequality on the $D$ norm and then
the definition of the induced matrix norm. 
\begin{eqnarray*}
\theta_{k} & \leq & \left\Vert Hx^{\left(k\right)}-H^{k}x^{\left(0\right)}\right\Vert _{D}+\left\Vert H^{k}x^{\left(0\right)}\right\Vert _{D}\\
 & = & \omega_{k}+\left\Vert H^{k}x^{\left(0\right)}\right\Vert _{D}\\
 & \leq & \omega_{k}+\left\Vert H\right\Vert _{D}^{k}\left\Vert x^{\left(0\right)}\right\Vert _{D}
\end{eqnarray*}
We now bound the error on the $\omega_{k}$ sequence: 
\begin{eqnarray*}
\omega_{k+1} & = & \left\Vert x^{\left(k+1\right)}-Hx^{\left(k\right)}+Hx^{\left(k\right)}-H^{k+1}x^{\left(0\right)}\right\Vert _{D}\\
 & \leq & \left\Vert Hx^{\left(k\right)}-H^{k+1}x^{\left(0\right)}\right\Vert _{D}+\left\Vert x^{\left(k+1\right)}-Hx^{\left(k\right)}\right\Vert _{D}\\
 & \leq & \left\Vert H\right\Vert _{D}\left\Vert x^{\left(k\right)}-H^{k}x^{\left(0\right)}\right\Vert _{D}+\nu\left\Vert Hx^{\left(k\right)}\right\Vert _{D}\\
 & = & \left\Vert H\right\Vert _{D}\omega_{k}+\nu\theta_{k}\\
 & \leq & \left\Vert H\right\Vert _{D}\omega_{k}+\nu\left(\omega_{k}+\left\Vert H\right\Vert _{D}^{k}\left\Vert x^{\left(0\right)}\right\Vert _{D}\right)\\
 & \leq & \left[\left\Vert H\right\Vert _{D}+\nu\right]\omega_{k}+\nu\left\Vert H\right\Vert _{D}^{k}\left\Vert x^{\left(0\right)}\right\Vert _{D}
\end{eqnarray*}
The assumption $\rho\leq1-\sqrt{\left\Vert H\right\Vert _{D}}$ is
equivalent to $\left\Vert H\right\Vert _{D}\leq\left(1-\rho\right)^{2}$,
so the last inequality implies: 
\[
\omega_{k+1}\leq\left[\left(1-\rho\right)^{2}+\nu\right]\omega_{k}+\nu\left(1-\rho\right)^{2k}\left\Vert x^{\left(0\right)}\right\Vert _{D}
\]

Note that the inequality $\left(1-\rho\right)^{2}+\nu\leq1-\rho$
is equivalent to $\nu\leq\rho-\rho^{2}$. Using the hypothesis, this
implies: 
\begin{equation}
\omega_{k+1}\leq\left(1-\rho\right)\omega_{k}+\nu\left(1-\rho\right)^{2k}\left\Vert x^{\left(0\right)}\right\Vert _{D}\label{eq:anon1}
\end{equation}

We show by induction that: 
\[
\forall k,\omega_{k}\leq\frac{\nu\left\Vert x^{\left(0\right)}\right\Vert _{D}}{1-\sqrt{1-\rho}}\left(\sqrt{1-\rho}\right)^{k-1}
\]

Note first that 
\begin{align*}
\omega_{1} & =\left\Vert x^{\left(1\right)}-Hx^{\left(0\right)}\right\Vert _{D}\\
 & \leq\nu\left\Vert Hx^{\left(0\right)}\right\Vert _{D}\\
 & \leq\nu\left\Vert H\right\Vert _{D}\left\Vert x^{\left(0\right)}\right\Vert _{D}\\
 & \leq\nu\left(1-\rho\right)^{2}\left\Vert x^{\left(0\right)}\right\Vert _{D}\\
 & \leq\nu\left\Vert x^{\left(0\right)}\right\Vert _{D}
\end{align*}
So this relation is verified for $k=1$. Now, assuming it is true
for $k$, we use Equation (\ref{eq:anon1}) to see that: 
\begin{eqnarray*}
\omega_{k} & \leq & \left(1-\rho\right)\omega_{k}+\nu\left(1-\rho\right)^{2k}\left\Vert x^{\left(0\right)}\right\Vert _{D}\\
 & \leq & \left(1-\rho\right)\omega_{k}+\nu\left(\sqrt{1-\rho}\right)^{k}\left\Vert x^{\left(0\right)}\right\Vert _{D}\\
 & \leq & \nu\left\Vert x^{\left(0\right)}\right\Vert _{D}\left[\frac{\left(1-\rho\right)}{1-\sqrt{1-\rho}}\left(\sqrt{1-\rho}\right)^{k-1}+\left(\sqrt{1-\rho}\right)^{k}\right]\\
 & = & \nu\left\Vert x^{\left(0\right)}\right\Vert _{D}\left(\sqrt{1-\rho}\right)^{k}\left[\frac{\sqrt{1-\rho}}{1-\sqrt{1-\rho}}+1\right]\\
 & = & \frac{\nu\left\Vert x^{\left(0\right)}\right\Vert _{D}}{1-\sqrt{1-\rho}}\left(\sqrt{1-\rho}\right)^{k}
\end{eqnarray*}
which is the the property for $k+1$. Using this property, we can
sum all the errors by a geometric series (note that $\omega_{0}=0$).
\[
\sum_{k=1}^{\infty}\omega_{k}\leq\frac{\nu\left\Vert x^{\left(0\right)}\right\Vert _{D}}{1-\sqrt{1-\rho}}\sum_{k=0}^{\infty}\left(\sqrt{1-\rho}\right)^{k}=\frac{\nu\left\Vert x^{\left(0\right)}\right\Vert _{D}}{\left(1-\sqrt{1-\rho}\right)^{2}}
\]

Finally, note that for $\rho\in\left(0,1/2\right)$, the inequality
$\nu\leq\rho/2$ implies $\nu\leq\rho-\rho^{2}$. Furthermore, by
concavity of the square root function, we have $\sqrt{1-\rho}\leq1-\rho/2$
for $\rho\leq1$. Thus, $\left(1-\sqrt{1-\rho}\right)^{2}\geq\rho^{2}/4$
and we get our result. 
\end{proof}
We can use the bound on the norm of $A$ to compute bound the error
with a preconditioner: 
\begin{lemma}
\label{lem:partial-sequence-approximate}Consider $A,B$ with the
same hypothesis as above, $x_{0}\in\mathbb{R}^{n}$, and the additional
hypothesis $\nu\in\left(0,\frac{1}{2\kappa}\right)$ and $\kappa\ge2$,
and $\left(x_{u}\right)_{u}$ an $\nu-$approximate power sequence
for the operator $R$ with start vector $x_{0}$. Then: 
\[
\left|\sum_{i=1}^{l}\frac{1}{i}x_{0}^{T}R^{i}x_{0}-\sum_{i=1}^{l}\frac{1}{i}x_{0}^{T}x_{i}\right|\leq4\nu\kappa^{2}\sqrt{\kappa\left(B\right)}\left\Vert x_{0}\right\Vert ^{2}
\]
where $\kappa\left(B\right)$ is the condition number of $B$.\end{lemma}
\begin{proof}
Call $\hat{z}$ the truncated sequence: 
\[
\hat{z}=\sum_{i=1}^{l}\frac{1}{i}x_{0}^{T}x_{i}
\]
This sequence is an approximation of the exact sequence $z$: 
\[
z=\sum_{i=1}^{l}\frac{1}{i}x_{0}^{T}R^{i}x_{0}
\]
We now bound the error between the two sequences: 
\begin{equation}
\left|\hat{z}-z\right|\leq\sum_{i=1}^{l}\frac{1}{i}\left|x_{0}^{T}\left(R^{i}x_{0}-x_{i}\right)\right|\leq\sum_{i=1}^{l}\left|x_{0}^{T}\left(R^{i}x_{0}-x_{i}\right)\right|\leq\sum_{i=1}^{l}\left|\left(B^{-1}x_{0}\right)^{T}B\left(R^{i}x_{0}-x_{i}\right)\right|\label{eq:anon}
\end{equation}
Using the Cauchy-Schwartz inequality, we obtain: 
\begin{equation}
\left|\left(B^{-1}x_{0}\right)^{T}B\left(R^{i}x_{0}-x_{i}\right)\right|=\left|\left\langle B^{-1}x_{0},R^{i}x_{0}-x_{i}\right\rangle _{B}\right|\leq\left\Vert B^{-1}x_{0}\right\Vert _{B}\left\Vert R^{i}x_{0}-x_{i}\right\Vert _{B}\label{eq:anon-1}
\end{equation}
From Lemma (\ref{lem:S-R-contractions}), we have $\left\Vert R\right\Vert _{B}\leq\left(1-\varphi\right)^{2}$,
and from the hypothesis, we have $\nu\in\left(0,\varphi/2\right)$
and $\varphi\le1/2$, so we can bound the deviation using the bound
from Lemma \ref{lem:nu-sequence-bound}: 
\begin{equation}
\sum_{i=1}^{l}\left\Vert R^{i}x_{0}-x_{i}\right\Vert _{B}\leq\sum_{i=1}^{\infty}\left\Vert R^{i}x_{0}-x_{i}\right\Vert _{B}\leq4\varphi^{-2}\nu\left\Vert x_{0}\right\Vert _{B}=4\kappa^{2}\nu\left\Vert x_{0}\right\Vert _{B}\label{eq:anon-2}
\end{equation}
Combining Equations (\ref{eq:anon}), (\ref{eq:anon-1}) and (\ref{eq:anon-2}),
we get: 
\[
\left|\hat{z}-z\right|\leq\left\Vert B^{-1}x_{0}\right\Vert _{B}\sum_{i=1}^{l}\left\Vert R^{i}x_{0}-x_{i}\right\Vert _{B}\leq4\nu\kappa^{2}\left\Vert B^{-1}x_{0}\right\Vert _{B}\left\Vert x_{0}\right\Vert _{B}
\]
Finally, it is more convenient to consider the Euclidian norm for
the norm of $x_{0}$. Call $\lambda_{\text{max }}$ and $\lambda_{\text{min }}$
the extremal eigenvalues of the positive semidefinite matrix $B$.
By definition of the matrix norm: $\left\Vert x_{0}\right\Vert _{B}=\sqrt{x_{0}^{T}Bx_{0}}\leq\sqrt{\lambda_{\text{max}}}\left\Vert x_{0}\right\Vert $
and $\left\Vert B^{-1}x_{0}\right\Vert _{B}=\sqrt{x_{0}^{T}B^{-1}x_{0}}\leq\sqrt{\lambda_{\text{min}}^{-1}}\left\Vert x_{0}\right\Vert $
so we get: 
\[
\left|\hat{z}-z\right|\leq4\nu\kappa^{2}\sqrt{\kappa\left(B\right)}\left\Vert x_{0}\right\Vert ^{2}
\]
where $\kappa\left(B\right)$ is the condition number of $B$. 
\end{proof}
We now have all the elements required for the proof of Theorem \ref{thm:preconditioning-approx}. 
\begin{proof}
Consider $\mathbf{u}_{j}\sim\mathcal{N}\left(0,I_{n}\right)$ for
$j=1\cdots p$, and $x_{i,j}=\begin{cases}
\mathbf{u}_{j}/\left\Vert \mathbf{u}_{j}\right\Vert  & \,\, i=0\\
x_{i-1,j}-C\left(Ax_{i-1,j}\right) & \,\, i>0
\end{cases}$

Call 
\[
z_{p,l}=\frac{1}{p}\sum_{j=1}^{p}\sum_{i=1}^{l}\frac{1}{i}\left(x_{0,j}\right)^{T}x_{i,j}
\]
\[
\hat{y}_{p,l}=\frac{1}{p}\sum_{j=1}^{p}\sum_{k=1}^{l}\frac{1}{k}\left(x_{0,j}\right)^{T}S^{k}x_{0,j}
\]

By construction, $\left(x_{i,j}\right)_{i}$ is an $\nu-$approximate
chain for the operator $R$. Applying Lemma \ref{lem:partial-sequence-approximate}
to the operator $R$ under the norm $B$, we get: 
\begin{eqnarray*}
\left|z_{p,l}-\hat{y}_{p,l}\right| & \leq & 4\nu\kappa^{2}\sqrt{\kappa\left(B\right)}\left[\frac{1}{p}\sum_{j=1}^{p}\left\Vert x_{0,j}\right\Vert ^{2}\right]=4\nu\kappa^{2}\sqrt{\kappa\left(B\right)}
\end{eqnarray*}

since $\left\Vert x_{0,j}\right\Vert ^{2}=1$, which gives us a deterministic
bound. Consider $\nu\leq\min\left(\frac{\epsilon}{8\kappa^{2}\sqrt{\kappa\left(B\right)}},\frac{1}{2\kappa}\right)$.
Then $\left|z_{p,l}-\hat{y}_{p,l}\right|\leq\epsilon/2$. Furthermore:
\[
\left|z_{p,l}-y\right|\leq\left|z_{p,l}-\hat{y}_{p,l}\right|+\left|y-\hat{y}_{p,l}\right|
\]
and $\mathbb{P}\left[\left|y-\hat{y}_{p,l}\right|\geq\epsilon/2\right]\leq\eta$
for a choice of $p\geq16\left(\frac{1}{\epsilon}+\frac{1}{n\epsilon^{2}}\right)\log\left(2/\eta\right)\log^{2}\left(\delta^{-1}\right)$
and $l\geq4\kappa\log\left(\frac{n}{\delta\epsilon}\right)$. Hence,
we get our bound result of 
\[
pl=64\kappa\left(\frac{1}{\epsilon}+\frac{1}{n\epsilon^{2}}\right)\log\left(2/\eta\right)\log^{2}\left(\delta^{-1}\right)\log\left(\frac{n}{\delta\epsilon}\right)
\]
\end{proof}

\section*{Appendix B: Proofs of Section \ref{sec:Making-the-problem}}

.\\

We put here the proofs that pertain to Section \ref{sec:Making-the-problem}.

\subsection{Properties of the generalized Laplacian}

Proof of Lemma \ref{lem:floating-properties}. 
\begin{proof}
The first statement is obvious from the construction of the grounded
Laplacian.

Statement (2) is a direct consequence of the fact that $F_{Z}=PZP^{T}$
with $P=\left(I_{n}\,0\right)$.

Then the third statement is a simple consequence of statement 2, as
$\text{ld}\left(Z\right)=\sum_{i}\log\lambda_{i}$ with $\left(\lambda_{i}\right)_{i}$
the $n-1$ positive eigenvalues of $Z$.

Statement (4) is straightforward after observing that the floating
procedure is a linear transform from $\mathcal{S}_{n}$ to $\mathcal{S}_{n-1}$,
so it preserves the matrix inequalities. 
\end{proof}

\subsection{Technical lemmas for Theorem \ref{thm:ultra_main}}

This lemma generalizes Lemma 8.1 in \cite{Spielman2009a}. 
\begin{lemma}
\label{lem:simple-inequality}Consider $A\in\mathcal{S}_{n}$ positive
semi-definite, and $x\in\mathbb{R}^{n}$. Then $xx^{T}\preceq\left(x^{T}A^{+}x\right)A$\end{lemma}
\begin{proof}
Without loss of generality, consider $x^{T}x=1$. Consider the eigenvalue
decomposition of $A$: $A=\sum_{i}\lambda_{i}u_{i}u_{i}^{T}$. Since
$\left(u_{i}\right)_{i}$ is an orthonormal basis of $\mathbb{R}^{n}$,
we only need to establish that $\left(u_{i}^{T}x\right)^{2}\leq\left(x^{T}A^{+}x\right)u_{i}^{T}Au_{i}$
for all $i$. The latter term can be simplified: 
\begin{eqnarray*}
\left(x^{T}A^{+}x\right)u_{i}^{T}Au_{i} & = & \left(x^{T}\left[\sum_{j}\lambda_{j}^{-1}u_{j}u_{j}^{T}\right]x\right)\lambda_{i}\\
 & = & \lambda_{i}\sum_{j}\lambda_{j}^{-1}\left(u_{j}^{T}x\right)^{2}\\
 & \geq & \left(u_{i}^{T}x\right)^{2}
\end{eqnarray*}
which is the inequality we wanted.\end{proof}
\begin{lemma}
\label{lem:Jensen-inequality-matrix-logarithm-1}Jensen inequality
for the matrix logarithm. Let $A\in\mathcal{S}_{n}$ be a positive
semi-definite matrix with $p$ positive eigenvalues. Then 
\[
\text{ld}\left(A\right)\leq p\log\left(\frac{\text{Tr}\left(A\right)}{p}\right)
\]
\end{lemma}
\begin{proof}
This is a direct application of Jensen's inequality. Call $\left(\lambda_{i}\right)_{i}$
the positive eigenvalues of $A$. Then $\text{ld}\left(A\right)=\sum_{i}\log\lambda_{i}$.
By concavity of the logarithm: 
\[
\sum_{i}\log\lambda_{i}\leq p\log\left(\frac{\sum\lambda_{i}}{p}\right)=p\log\left(\frac{\text{Tr}\left(A\right)}{p}\right)
\]
\end{proof}

\bibliographystyle{plain}
\bibliography{2012_sparse_info}

\begin{thebibliography}{10}

\bibitem{Abraham2008}
Ittai Abraham, Yair Bartal, and Ofer Neiman.
\newblock {Nearly tight low stretch spanning trees}.
\newblock {\em Foundations of Computer \ldots}, pages 781--790, 2008.

\bibitem{Alon1995}
Noga Alon, RM~Karp, D~Peleg, and Douglas West.
\newblock {A graph-theoretic game and its application to the k-server problem}.
\newblock {\em SIAM Journal on Computing}, 24(1):78--100, 1995.

\bibitem{atkins2011molecular}
Peter~W Atkins and Ronald~S Friedman.
\newblock {\em Molecular quantum mechanics}.
\newblock Oxford university press, 2011.

\bibitem{Bai1996}
Zhaojun Bai, Mark Fahey, and Gene~H. Golub.
\newblock {Some large-scale matrix computation problems}.
\newblock {\em Journal of Computational and Applied \ldots}, 74(1):71--89,
  1996.

\bibitem{Barry1999}
Ronald~Paul Barry and R.~Kelley Pace.
\newblock {Monte Carlo estimates of the log determinant of large sparse
  matrices}.
\newblock {\em Linear Algebra and its Applications}, 1999.

\bibitem{bernardson1994monte}
Shannon Bernardson, Paul McCarty, and Chris Thron.
\newblock Monte carlo methods for estimating linear combinations of inverse
  matrix entries in lattice \{QCD\}.
\newblock {\em Computer Physics Communications}, 78(3):256 -- 264, 1994.

\bibitem{deForcrand1989516}
Philippe de~Forcrand and Rajan Gupta.
\newblock Multigrid techniques for quark propagator.
\newblock {\em Nuclear Physics B - Proceedings Supplements}, 9(0):516 -- 520,
  1989.

\bibitem{duff1986direct}
Iain~S Duff, Albert~Maurice Erisman, and John~Ker Reid.
\newblock {\em Direct methods for sparse matrices}.
\newblock Clarendon Press Oxford, 1986.

\bibitem{duncan1998efficient}
A~Duncan, E~Eichten, and H~Thacker.
\newblock Efficient algorithm for qcd with light dynamical quarks.
\newblock {\em Physical Review D}, 59(1):014505, 1998.

\bibitem{Gremban1996}
Keith~D. Gremban.
\newblock {\em {Combinatorial Preconditioners for Sparse, Symmetric, Diagonally
  Dominant Linear Systems}}.
\newblock PhD thesis, Carnegie Mellon University, 1996.

\bibitem{Ipsen2006}
Ilse C~F Ipsen and Dean~J Lee.
\newblock {Determinant approximations}.
\newblock {\em Numerical Linear Algebra with Applications (under \ldots}, (X),
  2006.

\bibitem{kelner2013simple}
Jonathan~A Kelner, Lorenzo Orecchia, Aaron Sidford, and Zeyuan~Allen Zhu.
\newblock A simple, combinatorial algorithm for solving sdd systems in
  nearly-linear time.
\newblock In {\em Proceedings of the 45th annual ACM symposium on Symposium on
  theory of computing}, pages 911--920. ACM, 2013.

\bibitem{Koutis2010}
Ioannis Koutis, Gary~L Miller, and Richard Peng.
\newblock {Approaching optimality for solving SDD linear systems}.
\newblock pages 1--16, 2010.

\bibitem{Koutis2011}
Ioannis Koutis, Gary~L Miller, and Richard Peng.
\newblock {A nearly-m log n time solver for SDD linear systems}.
\newblock pages 1--16, 2011.

\bibitem{li2005analysis}
Runze Li and Agus Sudjianto.
\newblock Analysis of computer experiments using penalized likelihood in
  gaussian kriging models.
\newblock {\em Technometrics}, 47(2), 2005.

\bibitem{liu1990eliminationtrees}
J~Liu.
\newblock {The Role of Elimination Trees in Sparse Factorization}.
\newblock {\em SIAM Journal on Matrix Analysis and Applications},
  11(1):134--172, 1990.

\bibitem{lowdin1955quantum}
Per-Olov L{\"o}wdin.
\newblock Quantum theory of many-particle systems. iii. extension of the
  hartree-fock scheme to include degenerate systems and correlation effects.
\newblock {\em Physical review}, 97(6):1509, 1955.

\bibitem{martin1992approximations}
R~J Martin.
\newblock {Approximations to the determinant term in Gaussian maximum
  likelihood estimation of some spatial models}.
\newblock {\em Communications in Statistics-Theory and Methods},
  22(1):189--205, 1992.

\bibitem{McCourt2008}
M~McCourt.
\newblock {A Stochastic Simulation for Approximating the log-Determinant of a
  Symmetric Positive Definite Matrix}.
\newblock {\em compare}, 2:1--10, 2008.

\bibitem{meurant1999computer}
G\'{e}rard~A Meurant.
\newblock {\em {Computer Solution of Large Linear Systems}}.
\newblock North-Holland: Amsterdam, 1999.

\bibitem{KelleyPace1997291}
R.~Kelley Pace and Ronald Barry.
\newblock Sparse spatial autoregressions.
\newblock {\em Statistics and Probability Letters}, 33(3):291 -- 297, 1997.

\bibitem{Reusken2002}
Arnold Reusken.
\newblock {Approximation of the Determinant of Large Sparse Symmetric Positive
  Definite Matrices}.
\newblock {\em SIAM Journal on Matrix Analysis and Applications}, 23(3):799,
  2002.

\bibitem{Spielman2010}
Daniel~A Spielman.
\newblock {Algorithms , Graph Theory , and Linear Equations in Laplacian
  Matrices}.
\newblock Technical report, Proceedings of the International Congress of
  Mathematicians, Hyderabad, India, 2010.

\bibitem{Spielman2009}
Daniel~A Spielman and Nikhil Srivastava.
\newblock {Graph Sparsification by Effective Resistances}.
\newblock {\em SIAM Journal on Computing}, 40(6):1913--1926, 2011.

\bibitem{Spielman2008}
Daniel~A Spielman and Shang-Hua Teng.
\newblock {A Local Clustering Algorithm for Massive Graphs and its Application
  to Nearly-Linear Time Graph Partitioning}.
\newblock 2008.

\bibitem{Spielman2009a}
Daniel~A Spielman and Shang-Hua Teng.
\newblock {Nearly-Linear Time Algorithms for Preconditioning and Solving
  Symmetric , Diagonally Dominant Linear Systems}.
\newblock pages 1--48, 2009.

\bibitem{Spielman2009b}
Daniel~A Spielman and Jaeoh Woo.
\newblock {A Note on Preconditioning by Low-Stretch Spanning Trees}.
\newblock pages 1--4, 2009.

\bibitem{Wainwright2006}
M.J. Wainwright and M.I. Jordan.
\newblock {Log-determinant relaxation for approximate inference in discrete
  Markov random fields}.
\newblock {\em IEEE Transactions on Signal Processing}, 54(6):2099--2109, June
  2006.

\bibitem{Zhang2005}
Fuzhen Zhang.
\newblock {\em {The Schur complement and its applications}}.
\newblock 2005.

\bibitem{zhang2010kriging}
Hao Zhang and Yong Wang.
\newblock Kriging and cross-validation for massive spatial data.
\newblock {\em Environmetrics}, 21(3-4):290--304, 2010.

\bibitem{Zhang2007}
Y.~Zhang and W.~E. Leithead.
\newblock {Approximate implementation of the logarithm of the matrix
  determinant in Gaussian process regression}.
\newblock {\em Journal of Statistical Computation and Simulation},
  77(4):329--348, April 2007.

\bibitem{Zhang2008}
Yunong Zhang, W.E. Leithead, D.J. Leith, and L.~Walshe.
\newblock {Log-det approximation based on uniformly distributed seeds and its
  application to Gaussian process regression}.
\newblock {\em Journal of Computational and Applied Mathematics},
  220(1-2):198--214, October 2008.

\end{thebibliography}

\end{document}